%% file: dual_conf.tex
\documentclass[11pt]{amsart}

\usepackage{multirow}
\usepackage{amssymb}
\usepackage[bw,framed]{includes/mcode}
\usepackage{graphicx,color}
\usepackage{subfig}
\graphicspath{{figs/}}
\include{includes/sb_macros}
\newtheorem{assumption}[definition]{Assumption}

\usepackage{boxedminipage}

\def\Rinf{\R\cup \{\infty\}}

\def\hD{\widehat{D}}

\def\heta{\widehat{\eta}}
\def\hcI{\widehat{\cI}}

\usepackage{stmaryrd}

\begin{document}
\title[Primal-dual gap estimators]{Primal-dual gap estimators for a posteriori error analysis of nonsmooth minimization problems}
\author{S\"{o}ren Bartels}
\address{Department of Applied Mathematics, Mathematical Institute, University of Freiburg, Hermann-Herder-Str. 9, 79104 Freiburg i. Br., Germany}
\email{bartels@mathematik.uni-freiburg.de}
\author{Marijo Milicevic}
\address{Department of Applied Mathematics, Mathematical Institute, University of Freiburg, Hermann-Herder-Str. 9, 79104 Freiburg i. Br., Germany}
\email{marijo.milicevic@mathematik.uni-freiburg.de}
\date{\today}

\maketitle

\begin{abstract}
The primal-dual gap is a natural upper bound for the energy error
and, for uniformly convex minimization problems, also for the error in the energy norm.
This feature can be used to construct reliable primal-dual gap error estimators for which
the constant in the reliability estimate equals one for the energy error and equals the
uniform convexity constant for the error in the energy norm. In particular, it defines
a reliable upper bound for any functions that are feasible for the primal and the
associated dual problem. The abstract a posteriori error estimate based on the
primal-dual gap is provided in this article, and the abstract theory is applied to the
nonlinear Laplace problem and the Rudin-Osher-Fatemi image denoising problem.
The discretization of the primal and dual problems with conforming, low-order
finite element spaces is addressed. The primal-dual gap error estimator is used
to define an adaptive finite element scheme and numerical experiments are presented,
which illustrate the accurate, local mesh refinement in a neighborhood of the singularities,
the reliability of the primal-dual gap error estimator and the moderate
overestimation of the error.
\end{abstract}

\section{Introduction}

Many problems in various applications like partial differential equations, mechanics, imaging,
and operations research can be formulated as convex minimization problems of the form
\[
\inf_{u \in X} E(u) = \inf_{u \in X} F(Bu) + G(u)
\]
with convex functionals~$F,G$ and a bounded linear operator~$B$. Examples are
the nonlinear Laplace equation, the Rudin-Osher-Fatemi model for image denoising,
obstacle problems or convex programming. Depending on the data and the geometry
of the problem a solution $u \in X$ of the above minimization problem may suffer from
singularities which can harm the convergence rate as the mesh size $h>0$ of a
finite element method tends to zero.
A well-known example for this phenomenon is the linear Laplace problem on the L-shaped
domain. The geometry of the domain leads to a convergence rate of order~$\mathcal{O}(h^\g)$
instead of~$\mathcal{O}(h)$ in the energy norm, where $0<\g<1$ and~$\g$ depends on
the angle at the reentrant corner. Singularities may also arise due to intrinsic properties of
the functions in the underlying space~$X$. An example is the space of functions with
bounded variation~$BV(\O)$, which allows for jumps along interfaces, which is of interest, e.g.,
in image processing to preserve sharp edges. Yet, these jumps cause problems in the
finite element approximation of~$BV$-functions.

One way to overcome these drawbacks is adaptive mesh refinement. The general procedure
of adaptive routines is to compute an approximation of the minimizer in the discrete space
with a given underlying triangulation, compute a posteriori error estimators on the basis of the
computed approximation, refine the mesh locally where the error estimators are relatively
large and to compute a new approximate solution corresponding to the new mesh. In this
sense, adaptive methods are iterative numerical methods. The reader is referred to,
e.g.,~\cite{BabRhe78,AinOde97,NocSieVee09,Ver13,Ste07} to get an overview of adaptive
finite element methods.

The design of a posteriori error estimators is fundamental to adaptive finite element methods.
Particularly, it is crucial that the error estimators define upper (\textit{reliability})
and lower (\textit{efficiency}) bounds for an appropriate measure of the error and that
the constant in the upper bound is small and known. We will consider primal-dual gap
error estimators which can be derived using duality theory from convex analysis. In the
contributions~\cite{Rep00,Rep00:2,Rep00:3,Rep99,RepXan96,RepXan97,Bar15:2,BarSch17}
these primal-dual gap error estimators have been introduced and used
for various problems, e.g., elasto-plasticity and optimal transport. In~\cite{Rep00}
the primal-dual gap error estimator has been analyzed for general
convex minimization problems with uniformly convex functionals and the relation to other
a posteriori error estimators based on, e.g., residual and gradient recovery methods has been
addressed. Yet, the numerical study of primal-dual gap error estimators has not been considered
in any of those contributions.
We will analyze primal-dual gap based error estimators for the nonlinear Laplace problem
\[
E_{\Delta_\s}(u) = \frac{1}{\s} \int_\O |\nabla u|^\s \dv{x} - \int_\O fu \dv{x} \quad \longrightarrow \quad \text{Min.!}
\]
with $1<\s<\infty$, which has also been addressed in~\cite{Rep00:2} without a numerical study,
and for the Rudin-Osher-Fatemi (ROF) model
\[
E_{\rm rof}(u) = |\DD u|(\O) + \frac{\a}{2}\|u-g\|_{L^2(\O)}^2 \quad \longrightarrow \quad \text{Min.!}
\]
with~$|\DD u|(\O)$ the total variation of~$u$, which has been analyzed in, e.g.,~\cite{Bar15:2}.

The nonlinear Laplace problem serves as a model problem for degenerate nonlinear systems.
Results concerning the regularity of solutions, their approximation by finite elements and
a priori error estimates can be found, e.g.,
in~\cite{GloMar75,Cho89,BarLiu93,LiuBar93,LiuBar93:2,Ebm01,Ebm02,EbmLiuSte05,EbmLiu05}.
An important observation in the a priori error analysis was that the energy norm is not well suited
for the analysis since optimal convergence rates can only be guaranteed under restrictive assumptions
on the regularity of the solution, cf.~\cite{GloMar75,Cho89,BarLiu93,LiuBar93,LiuBar93:2}. It turned
out that a so-called \textit{quasi-norm}, which is a weighted $L^2$-norm of the gradient with
a weight depending on the gradient and which has been introduced in~\cite{BarLiu93},
is more appropriate for the analysis of the nonlinear Laplacian, cf.~\cite{EbmLiu05,DieRuz07}.
Particularly, the optimal convergence rate~$\mathcal{O}(h)$ for P1 finite elements can be proven under much less
restrictive regularity assumptions on the solution, cf.~\cite{EbmLiuSte05,EbmLiu05,DieRuz07}.
In~\cite{LiuYan01,LiuYan01:2,LiuYan02} residual-based a posteriori error estimators have been
proposed and reliability and efficiency has been established with respect to the quasi-norm.
However, the involved constants are not explicitly available. Residual-based
quasi-norm error estimators yielding explicit constants in the reliability estimate
have been discussed in~\cite{CarKlo03}
under the assumption that the modulus of the gradient is greater than zero almost everywhere
in the domain whereas the reliability and efficiency of quasi-norm error estimators
based on gradient recovery techniques has been established in~\cite{CarLiuYan06}.
The convergence of an adaptive scheme with residual-based a posteriori error estimators has
been proven in~\cite{Vee02}.
In~\cite{DieKre08,BelDieKre12} the linear convergence and optimality of an adaptive
method driven by residual-based quasi-norm error estimators has been proven.
The involved constants particularly for the upper bound depend on the nonlinearity of the problem.
In~\cite{AlaErnVoh11,ErnVoh13} the error is measured in a residual flux-based dual norm
and the a posteriori error estimator consists of a residual term, a diffusive flux term and a
linearization term. Flux reconstruction techniques are presented to compute the error
estimator and reliability (with constant one) and efficiency (with a constant independent of
the nonlinearity of the problem) are shown. Particular focus is on the balance of linearization
and discretization errors.

The ROF model serves as a prototype for $BV$-regularized minimization problems
with applications, e.g., in image processing (cf.~\cite{RudOshFat92,AubKor06}) and mechanics
(cf.~\cite{Tho13}). A primal-dual gap error estimator has been proposed to define an
adaptive algorithm for the ROF problem in~\cite{Bar15:2}, which has proven to
accurately detect the a priori unknown jump sets of the minimizer yielding locally refined
meshes in a neighborhood of the jump sets. Therein, a finite element method has been
proposed where the primal and dual problem
have been discretized with continuous, elementwise affine finite elements. However, the
approximation of the dual ROF problem by continuous finite elements is suboptimal since
the dual ROF problem is posed on~$H_\NN(\diver;\O)$. This is reflected in the experiments
in ~\cite{Bar15:2} where oscillations of the approximations along the interface can be observed.\\
\\
The advantage of primal-dual gap error estimators is that they are applicable to a
large class of convex minimization problems and naturally yield
upper bounds for the energy difference between the energy of an arbitrary admissible
test function and the optimal energy with constant one. In case of~$F$ or~$G$
being strongly convex (or coercive) they also define upper bounds for some appropriate error
measure with a constant depending on the coercivity constant. Particularly, they define
reliable upper bounds independently of the iterative solver used to approximate
discrete solutions to the primal and dual problem, i.e., the primal-dual gap error
estimator can be evaluated at any two feasible functions for
the primal and the dual problem to obtain an upper bound for the error. Last but not least,
the functionals~$F$ and~$G$ need not be assumed to be differentiable and there
does not need to exist a variational formulation of the primal problem to establish
the reliability of the primal-dual gap error estimators.\\
\\
In this paper we will consider primal-dual gap error estimators for both the nonlinear
Laplace problem and the ROF problem. While in~\cite{Rep00:2} the primal-dual gap
error estimator has been considered for the nonlinear Laplacian, the discretization and
numerical implementation is missing.
Furthermore, noting that the dual problem corresponding to the
nonlinear Laplace problem is given by a smooth, linearly constrained optimization problem
a modified error estimator, which is an upper bound for the primal-dual gap error estimator,
is suggested in~\cite{Rep00:2} allowing for dual test functions that do not satisfy the linear constraint.
We will consider the ``original'' primal-dual gap error estimator to control the quasi-norm
used in~\cite{BarLiu93,DieKre08}. In particular, the primal-dual gap error estimator~$\eta_{\rm pd}$ can be
used to improve the reliability estimate for the convergent, reliable and efficient residual-based error
estimator~$\eta_{\rm res}$ analyzed in~\cite{DieKre08,BelDieKre12}, i.e., defining
$\eta_{\rm com} = \min\{\eta_{\rm pd},\eta_{\rm res}\}$ we obtain a reliable, robust, efficient
and convergent error estimator. Continuous, piecewise affine finite elements are used
for the discretization of the primal nonlinear Laplace problem and the ROF problem
posed in~$W^{1,\a}(\O)$ and $BV(\O)\cap L^2(\O)$, respectively.
The dual problems are posed in $W^\b(\diver;\O)$, $\b=\a/(\a-1)$, and $H_\NN(\diver;\O)$ in case of
the nonlinear Laplacian and the ROF problem, respectively. In both cases
we use the Brezzi-Douglas-Marini finite element (cf. ~\cite{BofBreFor13}), which consists of
discontinuous piecewise affine vector fields with continuous normal components across
interelement sides, for the discretization. This is in contrast to the discretization
in~\cite{Bar15:2} where the dual ROF problem has been discretized with continuous,
piecewise affine vector fields, which is known to be problematic in, e.g., the discretization
of the dual formulation of the linear Laplacian with mixed finite elements. Particularly, oscillations are observed
in the approximation of~$u$ along the interface, cf. Section~\ref{sec:experiments}.
The discrete optimization problems related to the primal and the dual problems are
solved using the Variable-Alternating Direction Method of Multipliers (Variable-ADMM)
proposed in~\cite{BarMil17} which is an operator splitting method with variable step sizes.\\
\\
The paper is organized as follows. In Section~\ref{sec:preliminaries} we introduce the
notation, important function spaces and finite element spaces and state some approximation
results. The abstract primal-dual gap error estimator and a posteriori error estimate are
the subject of Section~\ref{sec:abstract error estimate}.
In Sections~\ref{sec:nonlinear Laplace} and~\ref{sec:ROF}
we state the nonlinear Laplace problem and the ROF problem, respectively,
and the associated dual problems, summarize a priori
and a posteriori error estimates and briefly address the numerical solution of the discrete
primal and dual problems. Finally, we present in Section~\ref{sec:experiments} our
numerical results for both problems for examples for which the exact solutions are explicitly
available.\\
\\
Let us remark that this article is part of the thesis~\cite{MIL19}, in which certain arguments
have been elaborated.

\section{Preliminaries}\label{sec:preliminaries}

\subsection{Function spaces and convex analysis}

We let $\O\subset\R^d$, ~$d=2,3$, be a bounded, polygonal
Lipschitz domain with Dirichlet boundary~$\G_\DD$ and Neumann boundary~$\G_\NN$
such that $\partial \O=\G_\DD \cup \G_\NN$. The $L^2$-norm on~$\O$
is denoted by~$\|\cdot\|$ and is induced by the scalar product
\[
(v,w) := \int_\O v \cdot w \dv{x}
\]
for scalar functions or vector fields $v,w \in L^2(\O;\R^r)$, $r \in \{1,d\}$, and we write~$|\cdot|$
for the Euclidean norm.\\
For $s \geq0$ and $\s \geq 1$ we let $W^{s,\s}(\O;\R^r)$ be the standard Sobolev space
with norm~$\|\cdot\|_{W^{s,\s}(\O)}$ and seminorm~$|\cdot|_{W^{s,\s}(\O)}$ with differentiability
exponent~$s$ and integrability exponent~$\s$. The subspace $W_\DD^{s,\s}(\O;\R^r)$ consists of
all functions in~$W^{s,\s}(\O;\R^r)$ that vanish on~$\G_\DD$ for $s\geq 1$ in the sense
of traces. If $s=0$ we write $L^\s(\O;\R^r)$ instead of~$W^{s,\s}(\O;\R^r)$.\\
Finally, for $\s'\geq 1$, we denote by $W^{\s'}(\diver;\O)$ the function space consisting of
all vector fields $p \in L^{\s'}(\O;\R^d)$ such that there exists a function $f \in L^{\s'}(\O)$ with
\[
\int_\O p \cdot \nabla \varphi \dv{x} = -\int_\O f \varphi \dv{x}
\]
for all continuously differentiable, compactly supported functions $\varphi \in C_c^1(\O)$.
If such a function $f \in L^{\s'}(\O)$ exists, we write $\diver p = f$. The space $W^{\s'}(\diver;\O)$
is equipped with the norm
\[
\|\cdot\|_{W^{\s'}(\diver;\O)}=\|\cdot\|_{L^{\s'}(\O)}+\|\diver\cdot\|_{L^{\s'}(\O)}.
\]
Furthermore, we denote by $W_\NN^{\s'}(\diver;\O)$ all elements of $p \in W^{\s'}(\diver;\O)$ with
$p\cdot n=0$ on~$\G_\NN$ in distributional sense, i.e.,
\[
\langle p\cdot n,u \rangle = \int_\O p \cdot \nabla u \dv{x} + \int_\O u \diver p \dv{x} = 0
\]
for all $u \in W_\DD^{1,\s}(\O)$, where $\s\geq 1$ is the dual exponent to $\s' \geq 1$, i.e.,
$1/\s + 1/\s'=1$. If $\s'=2$ we write $H(\diver;\O)$ instead of $W^2(\diver;\O)$, and accordingly
$H_\NN(\diver;\O)$ instead of $W_\NN^2(\diver;\O)$.\\
\\
For the general, abstract a posteriori error estimate we will work with
two reflexive Banach spaces~$X$ and~$Y$ equipped with the norms~$\|\cdot\|_X$
and~$\|\cdot\|_Y$, respectively. We denote their duals by~$X'$ and~$Y'$ and the
corresponding duality pairings by~$\langle \cdot,\cdot \rangle_{X',X}$ and~$\langle \cdot,\cdot \rangle_{Y',Y}$,
respectively. The double duals~$X''$ and~$Y''$ are identified with~$X$ and~$Y$, respectively.
If~$X$ is a Hilbert space with inner product~$(\cdot,\cdot)_X$, we identify the dual~$X'$
with~$X$. Given a bounded linear operator $B:X \to Y$ we denote by $B':Y'\to X'$ its adjoint.
For proper, convex and lower-semicontinuous functionals $F:Y \to \Rinf$ and $G:X \to \Rinf$
the subdifferentials $\p G(u)\subset X'$ at $u \in X$ and $\p F(p) \subset Y'$ at $p \in Y$ are defined by
\[\begin{split}
\partial G(u) &= \{w\in X': \; \langle w,v-u \rangle_{X',X} + G(u) \le G(v) \quad \text{for all } v \in X\},\\ 
\partial F(p) &= \{\l \in Y': \; \langle \l,q-p \rangle_{Y',Y} + F(p) \le F(q) \quad \text{for all } q \in Y\}.
\end{split}\]
Possible coercivity of the functionals~$F$ and~$G$ is 
characterized by non-negative mappings $\vrho_F:Y\times Y \to \R_+$ and 
$\vrho_G:X \times X \to \R_+$ such that for $w \in \partial G(u)$ and $\l \in \partial F(p)$ we have
\begin{equation}\label{eq:coercivity}
\begin{split}
\langle w,v-u \rangle_{X',X} + G(u) + \vrho_G(v,u) &\le G(v) \quad \text{for all } v \in X,\\
\langle \l,q-p \rangle_{Y',Y} + F(p)  + \vrho_F(q,p) &\le F(q) \quad \text{for all } q \in Y.
\end{split}
\end{equation}
This can be regarded as a generalization of the notion of uniform convexity
and strong convexity. The existence of non-trivial~$\vrho_G$ or~$\vrho_F$
will induce an error measure for which we establish primal-dual gap error
estimates. For the a posteriori error analysis we will need the Fenchel
conjugates~$F^*$ and~$G^*$, which are defined by
\[
F^*(q) = \sup_{p\in Y} \langle q,p \rangle_{Y',Y} - F(p), \quad 
G^*(v) = \sup_{u\in X} \langle v,u \rangle_{X',X} - G(u).
\]
These are used to convert the primal problems into dual problems.

\subsection{Finite element spaces}\label{subsec:fe}

We let~$(\cT_h)_{h>0}$ be a family of regular triangulations of~$\O$.
The set~$\cS_h$ consists of all edges ($d=2$)
or faces ($d=3$) of elements of~$\cT_h$ and~$\cN_h$ denotes the set of
nodes of~$\cT_h$. The elementwise constant mesh size function $h_\cT \in \cL^\infty(\O)$
is defined by
\[
h_\cT|_T = h_T = \operatorname{diam}(T)
\]
for all $T \in \cT_h$.
In the context of locally refined meshes we employ the average mesh size
\[
\overline{h} = |\cN_h|^{-1/d}
\]
defined with the cardinality~$|\cN_h|$ of~$\cN_h$.
Throughout the paper~$c$ will denote a generic, positive and mesh-independent constant.\\
For an integer $k \geq 0$ and a triangle $T \in \cT_h$ let~$P_k(T)$ be the space of
polynomials on~$T$ with total degree at most~$k$. We then consider for $r \in \{1,d\}$
the finite element spaces
\[
\cS^k(\cT_h)^r := \bigl\{v_h\in C(\overline{\O};\R^r): \; v_h|_T \in P_k(T)^r \text{ for all }T\in \cT_h\bigr\}
\]
and
\[
\cL^k(\cT_h)^r := \big\{q_h\in L^1(\O;\R^r): \; q_h|_T \in P_k(T)^r \text{ for all } T\in \cT_h\big\}.
\]
For an elementwise continuous function $v \in C(\cT_h)$ the operator
\[
\widehat{\cI}_h: C(\cT_h) \to \cL^1(\cT_h)
\]
is defined by the elementwise application of the standard nodal interpolation operator
$\cI_h: C(\overline{\O}) \to \cS^1(\cT_h)$. Note that $\widehat{\cI}_h|_{C(\overline{\O})}=\cI_h$.
With the nodal basis $\{\varphi_z: \; z \in \cN_h\} \subset \cS^1(\cT_h)$ the bilinear form
\[
(v,w)_h := \int_\O \widehat{\cI}_h(vw) \dv{x} = \sum_{T \in \cT_h} \sum_{z \in \cN_h\cap T} \b_z^T v|_T(z) w|_T(z)
\]
for $v,w \in \cL^1(\cT_h)$, where $\b_z=\int_T \varphi_z \dv{x}$, defines an inner product
on~$\cL^1(\cT_h)$. This mass lumping will allow for the
nodewise solution of certain nonlinearities. We have the relation
\[
\|v_h\| \leq \|v_h\|_h \leq (d+2)^{1/2} \|v_h\|
\]
for all $v_h \in \cL^1(\cT_h)$, cf.~\cite[Lemma 3.9]{Bar15}.

For completeness we provide the next lemma which states
that~$\cS^1(\cT_h)^d$ is dense in~$W^\b(\diver;\O)$.

\begin{lemma}\label{lem:density}
Let $p \in W^\b(\diver;\O)$. For every $\veps > 0$ there exists $h(\veps)>0$ such that
for all $h \leq h(\veps)$ there exists a function $q_h \in \cS^1(\cT_h)^d$ with
\[
\|p-q_h\|_{W^\b(\diver;\O)} < \veps.
\]
\end{lemma}

\begin{proof}
Since~$C^\infty(\overline{\O};\R^d)$ is dense in~$W^\b(\diver;\O)$, there exists for given
$\veps >0$ a function $q \in C^\infty(\overline{\O};\R^d)$ with
\[
\|p-q\|_{H(\diver;\O)} < \veps/2.
\]
Standard nodal interpolation estimates yield
\[
\|q-\cI_h q\|_{W^\b(\diver;\O)} \leq \|q-\cI_h q\|_{W^{1,\b}(\O;\R^d)} \leq c h |q|_{W^{2,\infty}(\O;\R^d)}.
\]
Now let~$h$ be such that
\[
\|q-\cI_h q\|_{W^\b(\diver;\O)} < \veps/2.
\]
Choosing $q_h=\cI_h q$ and using the triangle inequality yields the assertion.
\end{proof}

For an element $T \in \cT_h$ and $p_h \in P_k(T)^r$ we have by an
inverse estimate
\[
\|p_h\|_{L^2(T)}^2 \leq c h_T^{2\min\{0,d/2-d/\a\}}\|p_h\|_{L^\a(T)}^2,
\]
cf.~\cite{BreSco08}. Hence, we may introduce for $1\leq \a < 2$ the weighted $L^2$-inner product
\[
(p_h,q_h)_{w_\a} = (h_\cT^{d(2/\a - 1)}p_h,q_h)
\]
for $p_h,q_h \in \cL^k(\cT_h)$. Its induced norm then has the property
$\|\cdot\|_{w_\a} \leq c \|\cdot\|_{L^\a(\Omega)}$ on~$\cL^k(\cT_h)$.\\
\\
Let us finally introduce the so called \textit{Brezzi-Douglas-Marini (BDM)} finite element space
which is given by
\[
\mathcal{BDM}(\O)=\cL^1(\cT_h)^d \cap H(\diver;\O) \subset H(\diver;\O),
\]
cf.~\cite{BofBreFor13}. For an element $T \in \cT_h$ we can define a local
interpolation operator $\Pi_{h,T}: H^1(T)^d \to P_1(T)^d$ by
\[
\int_S q\cdot n \psi \dv{s} = \int_S \Pi_{h,T}q\cdot n \psi \dv{s}
\]
for all sides $S \in \cS_h \cap T$ of the element~$T$ and all affine functions $\psi \in P_1(S)$
on~$S$. Note that the interpolation operator is well-defined also for less regular functions, e.g.,
for $q \in H(\diver;T) \cap L^\g(T;\R^d)$ with $\g >2$, cf.~\cite{BofBreFor13}.
The global interpolation operator
$\Pi_h: H^1(\O)^d \to \mathcal{BDM}(\O)$ is then defined by
\[
(\Pi_h q)|_T = \Pi_{h,T}(q|_T)
\]
and, in particular, $\Pi_h q \in \mathcal{BDM}(\O) \subset H(\diver;\O)$.
For more details on $H(\diver;\O)$-conforming finite element spaces we refer the reader
to~\cite{BofBreFor13}.

\section{Abstract error estimate}\label{sec:abstract error estimate}

In the following we recap the existing results
on abstract a posteriori error estimation
for convex minimization problems and refer
to~\cite{Rep99,Rep00,Rep00:2,Bar15:2} for further details.

Let $F:Y \to \Rinf$ and $G:X \to \Rinf$ be proper, convex and lower-semicontinuous
functionals and $B:X \to Y$ be bounded and linear.
Under these hypothesis there holds $F=(F^*)^*$ and we obtain
\[\begin{split}
\inf_{u \in X} E(u) &= \inf_{u\in X} F(Bu) + G(u) \\
&= \inf_{u \in X} \sup_{p \in Y'} \; \langle p,Bu \rangle_{Y',Y} - F^*(p) + G(u) \\
&\geq \sup_{p \in Y'} \inf_{u \in X}  \; - F^*(p) + \langle p,Bu \rangle_{Y',Y} + G(u) \\
&= \sup_{p \in Y'} - \sup_{u \in X}  \;  F^*(p) + \langle -B'p,u \rangle_{X',X} - G(u)\\
&= \sup_{p \in Y'} \;  -F^*(p) - G^*(-B'p)\\
&=: \sup_{p \in Y'} D(p).
\end{split}\]
Hence, the dual formulation seeks a maximizer $p \in Y'$ for~$D$.
Particularly, we have the weak duality relation 
\begin{equation}\label{eq:weak duality}
E(v) \ge D(q)
\end{equation}
for all $v \in X$ and $q\in Y'$. If $u \in X$ is a minimizer for~$E$,
the necessary optimality condition reads
\[
0 \in \p E(u).
\]
With a nonnegative coercivity functional $\vrho_E:X \times X \to [0,\infty)$
this is equivalent to
\begin{equation}\label{eq:a priori}
\vrho_E(v,u) + E(u) \leq E(v)
\end{equation}
for all $v \in X$. A combination of~\eqref{eq:weak duality} and~\eqref{eq:a priori}
yields the following abstract a posteriori error estimate.

\begin{proposition}[Primal-dual gap estimates]\label{prop:abstract a posteriori}
Let $X_h \subset X$ and $Y_h \subset Y'$ and $u\in X$ and $u_h\in X_h$
be minimial for~$E$ in~$X$ and~$X_h$, respectively. We then have the a priori
error estimate
\[
\vrho_E(u,u_h) \le E(u_h)-E(u) \le \inf_{v_h\in X_h} E(v_h) - E(u).
\]
For any $w_h \in X_h$ and $q_h\in Y_h$ we have with $\eta(w_h,q_h):=(E(w_h)-D(q_h))^{1/2}$
the a posteriori error estimate
\[
\vrho_E(u,u_h) \le \eta^2(w_h,q_h).
\]
\end{proposition}

\begin{proof}
The a priori error estimate is a direct consequence of~\eqref{eq:a priori}. Using the
optimality~\eqref{eq:a priori} of $u \in X$, the weak duality~\eqref{eq:weak duality}
and $Y_h \subset Y'$ we then obtain
\[
\vrho_E(u,u_h) \le E(w_h) - E(u) \leq E(u_h) - \sup_{p \in Y'} D(p) \leq E(w_h) - D(q_h),
\]
which concludes the proof.
\end{proof}

\begin{remarks}
1. Note that in case of strong duality, i.e., there holds equality
in~\eqref{eq:weak duality}, the a posteriori error estimate stated in
Proposition~\ref{prop:abstract a posteriori} is sharp in the sense that if we use
$w_h=u$ and $q_h=p$ in~$\eta$ with $u\in X$ and $p \in Y'$ being
solutions to the primal and the dual problem, respectively, we have
\[
\eta^2(u,p) = E(u) - D(p) = \inf_{v \in X} E(v) - \sup_{q \in Y'} D(q) = 0.
\]
Sufficient for strong duality is that there exists $w \in X$ with $F(Bw)<\infty$,
$G(w)<\infty$ and~$F$ being continuous at~$Bw$. In this case the solutions
are related by the inclusions
\[
-B'p \in \p G(u), \quad p\in \p F(Bu),
\]
cf.~\cite{EkeTem99}, which are equivalent to the variational inequalities
\[\begin{split}
\langle -B'p,v-u\rangle_{X',X} + \vrho_G(v,u) + G(u) &\leq G(v),\\
\langle p,Bv-Bu\rangle_{Y',Y} + \vrho_F(Bv,Bu) + F(Bu) &\leq F(Bv).
\end{split}\]
Adding both inequalities gives~\eqref{eq:a priori} with
\[
\vrho_E(v,u) = \vrho_F(Bv,Bu) + \vrho_G(v,u),
\]
which serves as an error measure.\\
2. Let us emphasize that for the derivation of the reliability estimate for the
primal-dual gap error estimator~$\eta$ we did not need to make any assumptions
on the differentiability of the functionals~$F$ and~$G$.\\
3. One is free in the construction of feasible functions $w_h \in X_h$ and
$q_h \in Y_h$ to define the error estimator~$\eta(w_h,q_h)$. We will use
the Variable-ADMM introduced in~\cite{BarMil17} to approximately solve
the primal and the dual problem for the nonlinear Laplace problem and the ROF problem.
However, feasible functions, e.g., for the
dual problem, may be constructed using other techniques like gradient recovery or
flux reconstruction techniques, if they are applicable for the specific problem. The relation
between primal-dual gap error estimators and other error estimators is discussed
in~\cite{Rep00} for a certain class of convex minimization problems.
\end{remarks}

\section{Nonlinear Laplace equation}\label{sec:nonlinear Laplace}

\subsection{Primal and dual formulation}

The nonlinear Laplace problem seeks for $\s\in (1,\infty)$, $\s'=\s/(\s-1)$,
$f \in L^{\s'}(\O)$, $g \in L^{\s'}(\GN)$, $\tu_\DD \in W^{1,\s}(\O)$ and $u_\DD=\tu_\DD|_\GD$
a function $u\in W^{1,\s}(\O)$ which is minimal for 
\[
E_{\Delta_\s}(u) = \frac{1}{\s} \int_\O |\nabla u|^\s \dv{x} 
- \int_\O f u \dv{x} - \int_\GN g u \dv{s} + I_{u_\DD}(u|_\GD).
\]
The indicator functional~$I_{u_\DD}$ encodes the 
boundary condition $u|_\GD = u_\DD$ on 
$\GD = \p\O \setminus \GN$. The minimization problem admits a unique minimizer, cf. ~\cite{GloMar75}.
Minimization problems of the above structure arise in various areas of interest, e.g., nonlinear
diffusion~\cite{Phi61}, nonlinear elasticity~\cite{AtkCha84}, and fluid mechanics~\cite{AtkJon74,BarNaj90}.

Let us make the following assumption that will simplify the presentation.

\begin{assumption}
For ease of presentation we restrict to the case $g=0$ and
$u_\DD=0$ in what follows. We then omit the indicator functional~$I_{u_\DD}(u|_\GD)$
in the definition of~$E_{\Delta_\s}$ and seek for a minimizer $u\in W_\DD^{1,\s}(\O)$ instead.
\end{assumption}

The dual nonlinear Laplace problem seeks 
$p\in W_\NN^{\s'}(\diver;\O)$ that maximizes the functional
\[
D_{\Delta_\s}(p) := -\frac{1}{\s'} \int_\O |p|^{\s'} \dv{x}- I_{\{f\}}(-\diver p).
\]

The following result (cf.~\cite[Thm. 1]{Rep00:2}) shows that the dual
nonlinear Laplace problem is in fact the dual problem to the primal nonlinear
Laplace problem in the sense of Fenchel duality. It further ensures the strong
duality between the primal and the dual nonlinear Laplace problem.

\begin{theorem}[Strong duality]\label{thm:strong duality pLaplace}
There exists a unique minimizer $u \in W_\DD^{1,\s}(\O)$ for~$E_{\Delta_\s}$ and a unique
maximizer $p \in W_\NN^{\s'}(\diver;\O)$ for~$D_{\Delta_\s}$. The functions~$u$ and~$p$
are related by $\diver p = -f$, $p=|\nabla u|^{\s-2}\nabla u$
(or, equivalently, $\nabla u = |p|^{\s'-2}p$) and
\[
E_{\Delta_\s}(u) = D_{\Delta_\s}(p).
\]
\end{theorem}

\begin{proof}
The assertion follows from standard arguments in duality theory,
cf.~\cite{Rep00:2,EkeTem99}.
\end{proof}

Next, we introduce suitable finite element spaces for the primal
and dual nonlinear Laplace problem.

\subsection{Finite element spaces and a priori estimates}

To make use of the primal-dual gap estimator we need to choose conforming finite element spaces
$X_h \subset W_\DD^{1,\s}(\O)$ and $Y_h \subset W_\NN^{\s'}(\diver;\O)$. We let
\[
X_h=\cS^1(\cT_h) \cap W_\DD^{1,\s}(\O), \quad Y_h = \mathcal{BDM}(\O) \cap W_\NN^{\s'}(\diver;\O).
\]
For $f_h \in \cL^0(\cT_h)$ being the elementwise $L^2$-projection of~$f$ we set
\[\begin{split}
E_{\Delta_\s}^h(u_h) &= \frac{1}{\s} \int_\O |\nabla u_h|^\s \dv{x} 
- \int_\O f_h u_h \dv{x},\\
D_{\Delta_\s}^h(p_h) &= -\frac{1}{\s'} \int_\O |p_h|^{\s'} \dv{x}
- I_{\{f_h\}}(-\diver p_h), \\
\hD_{\Delta_\s}^h(p_h) &= -\frac{1}{\s'} \int_\O \hcI_h|p_h|^{\s'} \dv{x}
- I_{\{f_h\}}(-\diver p_h).
\end{split}\]

As it has been found out in earlier contributions, the energy
norm~$\|\nabla \cdot \|_{L^\s(\O)}$ is not well suited for the a priori
and a posteriori error analysis for the nonlinear Laplacian, since
one obtains convergence rates that are not optimal for a discretization with linear
finite elements, cf. \cite{BarLiu93}. Instead,
for a fixed function $v \in W^{1,\s}(\O)$, a so called \textit{quasi-norm} defined by
\[
\|\nabla w\|_{(v,\s)}^2 := \int_\O (|\nabla v|+|\nabla w|)^{\s-2}|\nabla w|^2 \dv{x}
\]
has been introduced and widely used in the literature, cf.
\cite{BarLiu93,EbmLiu05,LiuYan01,LiuYan01:2,LiuYan02,CarLiuYan06,DieKre08,BelDieKre12}.
Defining
\[
V(\nabla v) := |\nabla v|^{\frac{\s-2}{2}}\nabla v
\]
it has been shown in \cite{DieRuz07,DieKre08} that there exist constants $c,C>0$
with
\[
c \|\nabla v-\nabla w\|_{(v,\s)}^2 \leq \|V(\nabla v)-V(\nabla w)\|^2 \leq C \|\nabla v-\nabla w\|_{(v,\s)}^2.
\]
The following a priori estimate for the quasi-norm has been shown
in~\cite[Lem. 5.2]{DieRuz07}.

\begin{proposition}[A priori estimate]\label{prop:a priori primal p-laplace}
Let~$u$ and~$u_h$ be the minimizers for~$E_{\Delta_\s}$ in~$W_\DD^{1,\s}(\O)$ and
in~$X_h$, respectively. Then we have
\[
\|V(\nabla u) - V(\nabla u_h)\| \leq c \inf_{v_h \in X_h} \|V(\nabla u) - V(\nabla v_h)\|.
\]
If the minimizer~$u$ additionally satisfies $V(\nabla u) \in W^{1,2}(\O;\R^d)$, there holds
\[
\|V(\nabla u) - V(\nabla u_h)\| \leq c \inf_{v_h \in X_h} \|V(\nabla u) - V(\nabla v_h)\|
\leq c h \|\nabla V(\nabla u)\|.
\]
\end{proposition}

\begin{proof}
A complete proof is given in~\cite{DieRuz07}.
\end{proof}

\begin{remark}
Under certain regularity assumptions on the data~$f$ and the boundary~$\p \O$ one can prove
$V(\nabla u) \in W^{1,2}(\O;\R^d)$, cf.~\cite{Ebmeyer05,EbmLiuSte05}. In general,
one may only expect $V(\nabla u) \in W^{1,\g}(\O;\R^d)$ for some $\g >1$ and, in this case,
\[
\|V(\nabla u) - V(\nabla u_h)\| \leq c h^s
\]
with $s=\min\{1,2-2/\g\}$, cf.~\cite[Rem. 5.2]{BelDieKre12}.
\end{remark}

To obtain an a posteriori error estimate in the style of Proposition~\ref{prop:abstract a posteriori}
we need to bound the error in the quasi-norm by the energy difference.
This is established in~\cite[Lem. 16]{DieKre08} for the difference between two
finite element solutions of the nonlinear Laplace problem on nested finite element
spaces.

\begin{proposition}[{\cite[Lem. 16]{DieKre08}}]\label{prop:uniform convexity pLaplace}
Let~$u \in W_\DD^{1,\s}(\O)$ be the unique minimizer of~$E_{\Delta_\s}$ and
$v_h \in X_h$ be arbitrary. Then we have
\[
c\|V(\nabla u)-V(\nabla v_h)\|^2 \leq E_{\Delta_\s}(v_h) - E_{\Delta_\s}(u).
\]
\end{proposition}

\begin{proof}
A proof is presented in~\cite[Lem. 16]{DieKre08}, where the error between two
minimizers~$u_h \in X_h$ and~$u_{h'} \in X_{h'}$ of~$E_{\Delta_\s}$ in nested spaces
$X_h \subset X_{h'} \subset W_\DD^{1,\s}(\O)$ is considered. However, the
minimality property of~$u_h$ is not used so that we may replace it by
any test function $v_h \in X_h$, see also~\cite[Lem. 3.2, Rem. 3.3]{BelDieKre12}.
We refer the reader to~\cite[Lem. 16]{DieKre08} for details.
\end{proof}

The previous proposition enables us to follow the arguments for the a posteriori
error analysis presented in the abstract setting.

\subsection{A posteriori estimate and error estimator}

By Proposition~\ref{prop:uniform convexity pLaplace} and the strong duality
ensured by Theorem~\ref{thm:strong duality pLaplace} we obtain an a
posteriori error estimate and an error estimator in the fashion
of Proposition~\ref{prop:abstract a posteriori},
which can be used for adaptive local mesh refinement. The next result
is a special case of Proposition~\ref{prop:abstract a posteriori} for the nonlinear
Laplace problem, where also the data approximation error is taken into account.

\begin{proposition}[A posteriori estimate]\label{prop:a posteriori pLaplace}
Let~$u$ and~$u_h$ be the unique minimizers for~$E_{\Delta_\s}$ in~$W_\DD^{1,\s}(\O)$
and for~$E_{\Delta_\s}^h$ in~$X_h$, respectively, and let~$p_h$
be the unique maximizer for~$D_{\Delta_\s}^h$ in~$Y_h$. Then we have for
any $v_h \in X_h$ and $q_h \in Y_h$ with $\diver q_h = -f_h$ that
\[
c\|V(\nabla u)-V(\nabla u_h)\|^2 \leq \eta_{\Delta_\s}^h(v_h,q_h)^2 + c\|p_h\|_{L^{\s'}(\O)}^{1/(\s-1)}\|f-f_h\|_{L^{\s'}(\O)},
\]
with $\eta_{\Delta_\s}^h(v_h,q_h)^2=E_{\Delta_\s}^h(v_h) - D_{\Delta_\s}^h(q_h)$.
\end{proposition}

\begin{proof}
By Proposition~\ref{prop:uniform convexity pLaplace}, the strong duality
given by Theorem~\ref{thm:strong duality pLaplace} and the optimality of~$u_h$
and~$p_h$ in~$X_h$ and~$Y_h$, respectively, we have
\[\begin{split}
c\|V(\nabla u)-V(\nabla u_h)\|^2 \leq & \; E_{\Delta_\s}(u_h)-E_{\Delta_\s}(u) \\
= & \; E_{\Delta_\s}(u_h)-D_{\Delta_\s}(p) \\
= & \; E_{\Delta_\s}^h(u_h) - D_{\Delta_\s}^h(p_h) \\
&+ E_{\Delta_\s}(u_h)-E_{\Delta_\s}^h(u_h) + D_{\Delta_\s}^h(p_h) - D_{\Delta_\s}(p) \\
\leq & \; \eta_{\Delta_\s}^h(v_h,q_h)^2 \\
&+ E_{\Delta_\s}(u_h)-E_{\Delta_\s}^h(u_h) + D_{\Delta_\s}^h(p_h) - D_{\Delta_\s}(p).
\end{split}\]
Using $\|u_h\|_{L^\s(\O)} \leq c \|\nabla u_h\|_{L^\s(\O)} \leq c\|f_h\|_{L^{\s'}(\O)}^{1/(\s-1)}$
the first data approximation error can be estimated by
\[
E_{\Delta_\s}(u_h)-E_{\Delta_\s}^h(u_h) = \int_\O u_h (f-f_h) \dv{x}
\leq c\|f_h\|_{L^{\s'}(\O)}^{1/(\s-1)}\|f-f_h\|_{L^{\s'}(\O)}.
\]
To estimate the second error involving the discretization of the dual functional
we will construct a function $\widetilde{p}_h \in W_\NN^{\s'}(\diver;\O)$ for
which~$D_{\Delta_\s}$ is finite, i.e., $\diver \widetilde{p}_h=-f$, and which
relates~$p$ and~$p_h$. Let $w^{(h)} \in W^{1,\s}(\O)$ be the unique weak
solution with vanishing mean of
\[
-\diver(|\nabla w^{(h)}|^{\s-2}\nabla w^{(h)}) = f-f_h, \quad |\nabla w^{(h)}|^{\s-2}\nabla w^{(h)}\cdot n=0 \text{ on } \partial \Omega
\]
and set $p^{(h)}=|\nabla w^{(h)}|^{\s-2}\nabla w^{(h)}$. Then we have
$p^{(h)} \in W_\NN^{\s'}(\diver;\O)$ with
\[
\|p^{(h)}\|_{L^{\s'}(\O)} \leq c \|f-f_h\|_{L^{\s'}(\O)}.
\]
For $\widetilde{p}_h=p_h+p^{(h)}$ there holds $-\diver \widetilde{p}_h=f$, i.e.,
$D_{\Delta_\s}(\widetilde{p}_h) < \infty$.
With the optimality of~$p$ and the monotonicity
\[
|a|^{\s'} - |b|^{\s'} \leq {\s'} |a|^{\s'-2} a\cdot (a-b)
\]
for $a,b \in \R^d$ we can then bound the error $D_{\Delta_\s}^h(p_h) - D_{\Delta_\s}(p)$ by
\[\begin{split}
D_{\Delta_\s}^h(p_h) - D_{\Delta_\s}(p) \leq & \; D_{\Delta_\s}^h(p_h) - D_{\Delta_\s}(\widetilde{p}_h)\\
\leq &\; \int_\O |p_h|^{\s'-2}p_h \cdot (p_h-\widetilde{p}_h) \dv{x}\\
\leq & \; \||p_h|^{\s'-1}\|_{L^\s(\O)}\|p^{(h)}\|_{L^{\s'}(\O)} \\
\leq & \; c \|p_h\|_{L^{\s'}(\O)}^{1/(\s-1)}\|f-f_h\|_{L^{\s'}(\O)},
\end{split}\]
which completes the proof.
\end{proof}

\begin{remarks}
1. In our numerical experiments below the sequence of discrete solutions to the dual nonlinear
Laplace problem $(p_h)_{h>0}$ remained bounded in $L^{\s'}(\O)$. Unfortunately, we were
not able to prove this theoretically in general.\\
2. In Proposition~\ref{prop:positivity estimator pLaplace}
we prove that the density of the estimator is nonnegative.
\end{remarks}

\begin{remarks}\label{rem:estimator pLaplace}
1. Note that the (discrete) primal-dual gap error estimator~$\eta_{\Delta_\s}^h$
defines for arbitrary $v_h \in X_h$ and $q_h \in Y_h$ with $\diver q_h = -f_h$
a reliable upper bound (up to data oscillations) for the error
in the quasi-norm, i.e., we do not need to compute exact discrete solutions~$u_h$
and~$p_h$ of the primal and dual nonlinear Laplace problem, respectively.\\
2. The proof of the reliability of the primal-dual gap error estimator
did not require any differentiability assumptions on~$E_{\Delta_\s}$ or a
variational formulation of the primal nonlinear Laplace problem.\\
3. Using integration by parts and $\diver q_h = -f_h$ we obtain the expression
\[
\eta_{\Delta_\s}^h(v_h,q_h)^2  = \int_\O \frac{1}{\s}|\nabla v_h|^\s + \frac{1}{\s'}|q_h|^{\s'} - q_h \cdot \nabla v_h\dv{x}.
\]
4. In our numerical experiments we will use the computable (lumped) discrete
primal-dual gap error estimator
\[
\heta_{\Delta_\s}^h(v_h,q_h)^2=E_{\Delta_\s}^h(v_h) - \hD_{\Delta_\s}^h(q_h).
\]
As before, integration by parts and the relation $\diver q_h = -f_h$ yield
\[
\heta_{\Delta_\s}^h(v_h,q_h)^2  = \int_\O \frac{1}{\s}|\nabla v_h|^\s + \frac{1}{\s'}\hcI_h|q_h|^{\s'} - q_h \cdot \nabla v_h\dv{x}.
\]
\end{remarks}

For $T \in \cT_h$ the local error indicator is given by restriction
of the global error estimator to the element~$T$. We have the following
nonnegativity result.

\begin{proposition}\label{prop:positivity estimator pLaplace}
Let for any $T \in \cT_h$ the local error indicator be defined by
\[\begin{split}
\eta_{\Delta_\s}^{h,T}(v_h,q_h)^2 &= \int_T \frac{1}{\s}|\nabla v_h|^\s + \frac{1}{\s'}|q_h|^{\s'} - q_h \cdot \nabla v_h\dv{x},\\
\heta_{\Delta_\s}^{h,T}(v_h,q_h)^2  &= \int_T \frac{1}{\s}|\nabla v_h|^\s + \frac{1}{\s'}\hcI_h|q_h|^{\s'} - q_h \cdot \nabla v_h\dv{x}.
\end{split}\]
Then we have for any $v_h \in X_h$ and $q_h \in Y_h$
\[
\heta_{\Delta_\s}^{h,T}(v_h,q_h) \geq \eta_{\Delta_\s}^{h,T}(v_h,q_h) \geq 0.
\]
\end{proposition}

\begin{proof}
Using that for an element $T \in \cT_h$ and $x \in T$ the mapping $x \mapsto |q_h(x)|^{\s'}$
is convex we conclude that $\hcI_h |q_h|^{\s'} \geq |q_h|^{\s'}$
on~$T$ since~$q_h|_T$ is affine, and, therefore,
\[
\heta_{\Delta_\s}^{h,T}(v_h,q_h)^2 \geq \eta_{\Delta_\s}^{h,T}(v_h,q_h)^2.
\]
Note that the integrand in the definition of~$\eta_{\Delta_\s}^{h,T}$
is nonnegative, because for arbitrary $b \in \R^d$
we have by Young's inequality
\[
(1/\s') |b|^{\s'} = \sup_{a\in \R^d} a\cdot b - (1/\s) |a|^\s.
\]
Particularly, we have
\[
\eta_{\Delta_\s}^{h,T}(v_h,q_h)^2 = \int_T \frac{1}{\s}|\nabla v_h|^\s + \frac{1}{\s'}|q_h|^{\s'} - q_h \cdot \nabla v_h\dv{x} \geq 0
\]
for every element $T \in \cT_h$. Hence, putting everything together, we arrive at
\[
\heta_{\Delta_\s}^{h,T}(v_h,q_h) \geq \eta_{\Delta_\s}^{h,T}(v_h,q_h) \geq 0
\]
for any $T \in \cT_h$.
\end{proof}

In the sequel we briefly discuss the explicit computation of the primal-dual
gap error estimator.

\subsection{Iterative solution}

As we have pointed out in Remark~\ref{rem:estimator pLaplace} the quantity~$\eta_{\Delta_\s}^h(v_h,q_h)$,
and therefore also~$\heta_{\Delta_\s}^h(v_h,q_h)$ by Proposition~\ref{prop:positivity estimator pLaplace},
defines a reliable upper bound for any feasible functions $v_h \in X_h$ and $q_h \in Y_h$.
Since the minimizer~$u_h$ of~$E_{\Delta_\s}^h$ in~$X_h$ and the maximizer~$p_h$
of~$D_{\Delta_\s}^h$ in~$Y_h$ are not directly available, a reasonable choice of
functions~$v_h$ and~$q_h$ with $\diver q_h = -f_h$ are approximate discrete
solutions of the primal and dual nonlinear Laplace problem. These will be computed
using splitting methods based on augmented Lagrange functionals, which have been
introduced in~\cite{GloMar75,GabMer76}. For the primal problem we define
\[\begin{split}
L_\tau^E (u_h,r_h;\l_h) &= \frac1\s \int_\O |r_h|^\s \dv{x} 
- \int_\O f_h u_h \dv{x} \\
&\quad + (\l_h, \nabla u_h - r_h)_{w_\s} + \frac{\tau}{2}\|\nabla u_h - r_h\|_{w_\s}^2
\end{split}\]
for $u_h \in X_h$ and $r_h,\l_h \in \cL^0(\cT_h)^d$.
For the dual problem we consider
\[\begin{split}
L_\tau^D(p_h,q_h;\mu_h) &= \frac{1}{\s'} \int_\O \hcI_h|q_h|^{\s'} \dv{x} + I_{\{-f_h\}}(\diver p_h) \\
&\quad + (\mu_h, p_h-q_h)_{h,w_{\s'}} + \frac{\tau}{2} \|p_h - q_h\|_{h,w_{\s'}}^2
\end{split}\]
for $q_h,\mu_h \in \cL^1(\cT_h)^d$ and $p_h \in Y_h$. The minimization
of~$E_{\Delta_\s}^h$ and~$-\hD_{\Delta_\s}^h$ is equivalent to seeking a saddle point
for~$L_\tau^E$ and~$L_\tau^D$, respectively, i.e.,
\[\begin{split}
\min_{u_h \in X_h} E_{\Delta_\s}^h(u_h) &= \min_{(u_h,r_h)\in X_h \times \cL^0(\cT_h)^d} \max_{\l_h \in \cL^0(\cT_h)^d} L_\tau^E (u_h,r_h;\l_h),\\
\min_{p_h \in Y_h} -\hD_{\Delta_\s}^h(p_h) &= \min_{(p_h,q_h)\in Y_h \times \cL^1(\cT_h)^d} \max_{\mu_h \in \cL^1(\cT_h)^d} L_\tau^D (p_h,q_h;\mu_h).
\end{split}\]
The associated saddle-point problems are then solved
using the Variable-ADMM, cf. ~\cite{BarMil17} for details.

%
%
%
%
%

\section{Rudin-Osher-Fatemi image denoising}\label{sec:ROF}

\subsection{Primal and dual formulation}

In this section we consider a variant of the nonlinear Laplacian with limit
exponent $\s=1$.
For a given function $g\in L^2(\O)$ and a fidelity parameter $\a>0$ we seek
a minimizer $u\in BV(\O)\cap L^2(\O)$ of the functional
\[
E_{\rm rof}(u) = \int_\O |\DD u| + \frac{\a}{2} \|u-g\|^2.
\]
This particular minimization problem has been proposed in image processing for denoising
a given noisy image~$g$ and is known as the Rudin-Osher-Fatemi (ROF) image denoising
problem~\cite{RudOshFat92}. It also serves as a model problem for
general $BV$-regularized minimization problems and evolutions, cf., e.g.,~\cite{Tho13}.
The (pre-)dual problem is given by the maximization of the functional
\[
D_{\rm rof}(p) = -\frac{1}{2\a} \|\diver p + \a g \|^2 + \frac{\a}{2} \|g\|^2 - I_{K_1(0)}(p)
\]
in the set of vector fields $p\in H_\NN(\diver;\O)$ with square integrable distributional
divergence and vanishing normal component on $\p\O$, cf.~\cite{KunHin04}.
The indicator functional~$I_{K_1(0)}$ of the set of vector fields $q\in L^2(\O;\R^d)$
which satisfy $|q|\le 1$ in~$\O$ introduces a pointwise constraint.
Note that a maximizer of~$D_{\rm rof}$ may not be unique. The primal and the dual ROF
problem are in strong duality and the unique minimizer
$u \in BV(\O)\cap L^2(\O)$ of~$E_{\rm rof}$ and any maximizer $p \in H_\NN(\diver;\O)$
of~$D_{\rm rof}$ are related by
\[
\diver p = \a (u-g), \quad -(u,\diver(q-p)) \leq 0
\]
for all $q \in H_\NN(\diver;\O) \cap K_1(0)$, cf.~\cite{KunHin04}.

\subsection{Finite element spaces and a priori estimates}

As for the nonlinear Laplace equation we let
\[
X_h = \cS^1(\cT_h)\subset BV(\O)\cap L^2(\O).
\]
The discrete space~$Y_h$ is chosen to consist of continuous or discontinuous,
elementwise affine vector fields
\[
Y_h^C=\cS^1(\cT_h)^d \cap H_\NN(\diver;\O), \quad \text{or} \quad Y_h^{dC} = \cL^1(\cT_h)^d \cap H_\NN(\diver;\O).
\]
We have the consistency relation $Y_h^C \subset Y_h^{dC} \subset H_\NN(\diver;\O)$
and denote by ~$Y_h$ either of the two spaces. Let $g_h \in \cL^0(\cT_h)$
be the elementwise $L^2$-projection of~$g$. The discretized functionals are then defined by
\[\begin{split}
E_{\rm rof}^h(u_h) &= \int_\O |\nabla u_h| \dv{x} + \frac{\a}{2}\|u_h-g_h\|^2,\\
D_{\rm rof}^h(p_h) &= -\frac{1}{2\a}\|\diver p_h + \a g_h\|^2 - I_{K_1(0)}(p_h) + \frac{\a}{2}\|g_h\|^2.
\end{split}\]

\begin{remark}
The discretization of the dual ROF problem with the lowest order Raviart-Thomas finite element
is not suitable since it does not include nodal degrees of freedom which is required to ensure
the pointwise constraint $|p_h|\leq 1$ which in turn is mandatory to derive a meaningful and
useful a posteriori error estimate.
\end{remark}

Let~$u$ and~$u_h$ be the unique minimizers of ~$E_{\rm rof}$ in $BV(\O) \cap L^2(\O)$
and~$X_h$, respectively. The strong convexity of~$E_{\rm rof}$ can be used to derive
the a priori error estimate
\[
\frac{\a}{2} \|u-u_h\|^2  \leq c h^{1/2}
\]
if $u \in BV(\O) \cap L^\infty(\O)$, cf.~\cite{BarNocSal14,Bar15}.
The optimal convergence rate for the approximation with continuous, piecewise linear
functions is, however, given by
\[
\min_{v_h\in \cS^1(\cT_h)} \|u-v_h\|^2 \le c h,
\]
which cannot be improved in general, cf.~\cite{BarNocSal14,Bar15}.\\
\\
Motivated by the relation $\diver p = \a (u-g)$ we also consider
for any discrete maximizer $p_h \in Y_h$ of~$D_{\rm rof}^h$
the approximation
\[
\overline{u}_h = \frac{1}{\a} \diver p_h + g_h \in \cL^0(\cT_h)
\]
of~$u$, for which the following convergence result can be proven.

\begin{proposition}\label{prop:piecewise constant approximation}
Let for any $h>0$ the function~$p_h$ be a discrete maximizer
of~$D_{\rm rof}^h$ in~$Y_h$ and let $\overline{u}_h = (1/\a) \diver p_h + g_h$.
If $g_h \to g$ in~$L^2(\O)$, we have
\[
\|u-\overline{u}_h\| \to 0
\]
as $h \to 0$.
\end{proposition}

\begin{proof}
The sequence $(g_h)_h \subset L^2(\O)$ is uniformly
bounded since $g_h \to g$ in~$L^2(\O)$.
Using that~$p_h$ is a minimizer for~$-D_{\rm rof}^h$ in~$Y_h$ we can bound
\[
\frac{1}{2\a}\|\diver p_h + \a g_h\|^2 - \frac{\a}{2}\|g_h\|^2=-D_{\rm rof}^h(p_h)
\leq -D_{\rm rof}^h(0) = 0,
\]
i.e.,
\[
\frac{1}{2\a}\|\diver p_h + \a g_h\|^2 \leq \frac{\a}{2}\|g_h\|^2.
\]
Thus, the sequence $(p_h)_{h>0}$ is uniformly bounded in $H_\NN(\diver;\O)$.
Hence, we can choose a subsequence~$(p_{h'})_{h'>0}$
with $p_{h'} \wto p$ for a function $p \in H_\NN(\diver;\O)$. On the other hand there exists
for any $q \in H_\NN(\diver;\O)$ a sequence $(q_h)_{h>0} \subset Y_h^C$ with
$|q_h|\leq 1$ for all $h>0$ and $q_h \to q$ in $H_\NN(\diver;\O)$.
Indeed, for given $q \in H_\NN(\diver;\O)$ one can construct a smooth function
$\widetilde{q} \in C_c^\infty(\O;\R^d)$ via
convolution of~$q$ with a nonnegative convolution kernel noting that this process does not increase
the $L^\infty$-norm. One then procedes as in the proof of Lemma~\ref{lem:density} noting again that
neither the nodal interpolation operator increases the $L^\infty$-norm. The weak
lower-semicontinuity of~$-D_{\rm rof}$ and the optimality of each~$p_{h'}$ yield
\[\begin{split}
-D_{\rm rof}(p) &\leq \liminf_{h'\to0} - D_{\rm rof}(p_{h'}) \\
 &\leq \limsup_{h'\to0} - D_{\rm rof}^{h'}(p_{h'}) +  D_{\rm rof}^{h'}(p_{h'}) - D_{\rm rof}(p_{h'})\\
&\leq \limsup_{h'\to0} - D_{\rm rof}^{h'}(p_{h'}) +  c\|g-g_{h'}\|\\
&\leq \limsup_{h' \to 0} -D_{\rm rof}^{h'}(q_{h'}) \\
&= \limsup_{h' \to 0} -D_{\rm rof}(q_{h'}) + D_{\rm rof}(q_{h'}) - D_{\rm rof}^{h'}(q_{h'})\\
&\leq \limsup_{h' \to 0} -D_{\rm rof}(q_{h'}) + c\|g-g_{h'}\| = -D_{\rm rof}(q).
\end{split}\]
Hence,~$p$ is a minimizer of~$-D_{\rm rof}$. By choosing a sequence $(q_h)_{h>0} \subset Y_h^C$
such that $q_h \to p$ in $H_\NN(\diver;\O)$ we find that
\[
-D_{\rm rof}(p) = \lim_{h' \to 0} -D_{\rm rof}(p_{h'}),
\]
and, in particular, since $g_{h'} \to g$,
\[
\|\diver p_{h'}\| \to \|\diver p\|.
\]
This implies that $\diver p_{h'} \to \diver p$ since $\diver p_{h'} \wto \diver p$.
By strong duality of the primal and dual ROF problem we have
\[
u=\frac{1}{\a}\diver p + g.
\]
With $\diver p_{h'} \to \diver p$ and $g_{h'} \to g$ it follows that
\[
u-\overline{u}_{h'} = \frac{1}{\a}\diver p + g - \frac{1}{\a}\diver p_{h'} - g_{h'} \to 0.
\]
Thus, every convergent subsequence of $(\overline{u}_h)_{h>0}$ converges
to~$u$. Therefore, the whole sequence converges to~$u$.
\end{proof}

Using the strong convexity of the functional~$E_{\rm rof}$, i.e., there holds
\begin{equation}\label{eq:coercivity ROF}
\frac{\a}{2}\|u-v_h\|^2 \leq E_{\rm rof}(v_h) - E_{\rm rof}(u)
\end{equation}
for any $v_h \in \cS^1(\cT_h)$, we can carry out the a posteriori error analysis.

\subsection{A posteriori estimate and error estimator}

By the strong convexity~\eqref{eq:coercivity ROF} and the strong duality of the primal and
dual ROF problem we can establish an a
posteriori error estimate and an error estimator in the fashion
of Proposition~\ref{prop:abstract a posteriori},
which can be used for adaptive mesh refinement. The following reliability result
is a special case of Proposition~\ref{prop:abstract a posteriori} for the ROF
problem, where also the data approximation error is taken into account.

\begin{proposition}\label{prop:a posteriori ROF}
Let~$u$ and~$u_h$ be the unique minimizers for~$E_{\rm rof}$ in $BV(\O)\cap L^2(\O)$
and~$E_{\rm rof}^h$ in~$X_h$, respectively, and let~$p_h$ be a maximizer
for~$D_{\rm rof}^h$ in~$Y_h$. Then we have for any $v_h \in X_h$ and
$q_h \in Y_h$ with $|q_h|\leq 1$ that
\[
\frac{\a}{2} \|u-u_h\|^2 \leq \eta_{\rm rof}^h(v_h,q_h)^2 + c\|g-g_h\|
\]
with $\eta_{\rm rof}^h(v_h,q_h)^2 = E_{\rm rof}^h(v_h) - D_{\rm rof}^h(q_h)$
and~$c$ depending on~$\|g\|$.
\end{proposition}

\begin{proof}
Let $p \in H_\NN(\diver;\O)$ be a maximizer of~$D_{\rm rof}$.
Taking $v=u_h$ in~\eqref{eq:coercivity ROF} and
using the strong duality $E_{\rm rof}(u) = D_{\rm rof}(p)$, the optimality of~$p$
in~$H_\NN(\diver;\O)$, the optimality of~$p_h$ in $Y_h \subset H_\NN(\diver;\O)$
and the optimality of~$u_h$ in~$X_h$ we have
\[\begin{split}
\frac{\a}{2} \|u-u_h\|^2 \leq & \; E_{\rm rof}(u_h) - E_{\rm rof}(u) \\
= & \; E_{\rm rof}(u_h) - D_{\rm rof}(p) \\
\leq & E_{\rm rof}(u_h) - D_{\rm rof}(p_h) \\
= & \; \eta_{\rm rof}^h(u_h,p_h)^2\\
&+ E_{\rm rof}(u_h) - E_{\rm rof}^h(u_h) + D_{\rm rof}^h(p_h) - D_{\rm rof}(p_h)\\
\leq & \; \eta_{\rm rof}^h(v_h,q_h)^2\\
&+ E_{\rm rof}(u_h) - E_{\rm rof}^h(u_h) + D_{\rm rof}^h(p_h) - D_{\rm rof}(p_h).
\end{split}\]
The first data approximation error can be bounded by
\[
E_{\rm rof}(u_h) - E_{\rm rof}^h(u_h) = \frac{\a}{2}\int_\O (g_h-g)(2u_h-g-g_h) \dv{x} \leq c \|g-g_h\|,
\]
where we used that $\|u_h\| \leq c \|g_h\|$ and $\|g_h\| \leq c \|g\|$. The second data
approximation error can be analogously estimated by
\[\begin{split}
& \; D_{\rm rof}^h(p_h) - D_{\rm rof}(p_h) \\
= & \; \frac{1}{2}\Bigl[\int_\O (g_h-g)(g_h+g) \dv{x} +\int_\O (g-g_h)(2\diver p_h + \a(g+g_h)) \dv{x}\Bigr] \\
\leq & \; c \|g-g_h\|
\end{split}\]
using that $\|\diver p_h\| \leq c \|g_h\| \leq c \|g\|$, which completes the proof.
\end{proof}

\begin{remarks}\label{rem:estimator ROF}
1. Note that, as for the nonlinear Laplace problem, the (discrete) primal-dual gap error estimator~$\eta_{\rm rof}^h$
defines for arbitrary $v_h \in X_h$ and $q_h \in Y_h$ with $|q_h|\leq 1$
a reliable upper bound (up to data oscillations) for the error. Particularly,
the exact discrete solutions~$u_h$
and~$p_h$ of the primal and dual ROF problem, respectively, need not to be computed
exactly to estimate the error.\\
2. Using binomial formulas
and integration by parts we obtain the representation
\[
\eta_{\rm rof}^h(v_h,q_h)^2 =\int_\O |\nabla v_h| -\nabla v_h \cdot q_h \dv{x}
+ \frac{1}{2 \a}\|\diver q_h - \a(v_h-g_h)\|^2
\]
for $v_h \in X_h$ and $q_h \in Y_h$ with $|q_h|\leq 1$.
\end{remarks}

As for the nonlinear Laplace problem, for $T \in \cT_h$ the local error indicators are defined via
restricting the global error estimator to the simplex~$T$. The local error indicators
are non-negative due to the condition $|q_h|\leq 1$ as the next proposition shows.

\begin{proposition}
Let for any $T \in \cT_h$ the local error indicator be defined by
\[
\eta_{\rm rof}^{h,T}(v_h,q_h)^2 = \int_T |\nabla v_h| -\nabla v_h \cdot q_h \dv{x}
+ \frac{1}{2 \a}\|\diver q_h - \a(v_h-g)\|_{L^2(T)}^2.
\]
Then we have for any $v_h \in X_h$ and $q_h \in Y_h$ with $|q_h|\leq 1$ that
\[
\eta_{\rm rof}^{h,T}(v_h,q_h) \geq 0.
\]
\end{proposition}

\begin{proof}
The non-negativity immediately follows from $|q_h|\leq 1$ and the Cauchy-Schwarz inequality.
\end{proof}

To obtain a computable a posteriori error estimator we iteratively solve the
primal and dual ROF problem.

\subsection{Iterative solution}

We approximate discrete minimizers ~$u_h$ and ~$p_h$ of ~$E_{\rm rof}^h$ and
~$-D_{\rm rof}^h$ as in the case of the nonlinear Laplacian via an augmented Lagrangian approach.
To this end, we introduce for the primal problem
\[\begin{split}
L_\tau^E (u_h,r_h;\l_h) = & \; \int_\O |r_h| \dv{x} + \frac{\a}{2}\|u_h-g_h\|^2 \\
 &+ (\l_h, \nabla u_h - r_h)_w + \frac{\tau}{2}\|\nabla u_h - r_h\|_w^2
\end{split}\]
for $u_h \in X_h$ and $r_h,\l_h \in \cL^0(\cT_h)^d$, and, for the dual problem,
\[\begin{split}
L_\tau^D (p_h,q_h;\mu_h) = & \; \frac{1}{2\a} \|\diver p_h + \a g_h \|^2 - \frac{\a}{2} \|g_h\|^2 + I_{K_1(0)}(q_h) \\
 &+ (\mu_h, p_h - q_h)_h + \frac{\tau}{2}\|p_h - q_h\|_h^2
\end{split}\]
for $q_h,\mu_h \in \cL^1(\cT_h)^d$ and $p_h \in Y_h$. The corresponding saddle-point
problems are again solved using the Variable-ADMM presented in~\cite{BarMil17}.

%
%

\section{Numerical experiments}\label{sec:experiments}

In this section we present our numerical results for the approximation of solutions
for the nonlinear Laplace equation and the ROF problem using mesh adaptivity which is
based on the primal-dual gap estimators ~$\eta(u_h,p_h)$. The refinement of a given
triangulation ~$\cT_h$ relies on the D\"orfler marking and consists in the bisection
of elements $T \in \mathcal{M}_h$ of a minimal set $\mathcal{M}_h \subset \cT_h$ for which
\[
\Bigl[\sum_{T \in \mathcal{M}_h} \eta^T(u_h,p_h)^2\Bigr]^{1/2} \geq 1/2 \Bigl[\sum_{T \in \cT_h} \eta^T(u_h,p_h)\Bigr]^{1/2}
\]
holds. Additional elements then are refined to avoid hanging nodes.
The numerical approximations ~$u_h$ and ~$p_h$ for the primal and dual problem,
respectively, are obtained using the corresponding saddle-point formulations and the
Variable-ADMM presented in ~\cite{BarMil17}.\\
Before we report the performance of the adaptive algorithm for the
nonlinear Laplace equation and the ROF problem in this section, we will first briefly
comment on the hybrid realization of the Brezzi-Douglas-Marini finite element space.

\subsection{Hybrid implementation of $\mathcal{BDM}(\O)$}

We first of all define the space
\[
Z_h = \big\{r_h \in L^\infty(\cup \cS_h) : r_h|_S \mbox{ affine for all } S\in \cS_h\big\},
\]
i.e., ~$Z_h$ contains all functions~$r_h$ that are piecewise affine, discontinuous functions
on the skeleton~$\cS_h$ of the triangulation~$\cT_h$.
The space~$\mathcal{BDM}(\O)$ consists of all elementwise affine vector fields~$q_h$ for which the normal
component is continuous across interelement sides $S\in \cS_h$, i.e.,  
\[
[\![q_h\cdot n_S]\!]|_S(x) = \lim_{\veps \to 0} \big(q_h(x+\veps n_S) - q_h(x-\veps n_S) \big)\cdot n_S = 0
\]
for all $x\in S$ with a unit normal $n_S$ on $S$. If~$\mathcal{BDM}(\O)$ is defined to be a subspace
of~$H_\NN(\diver;\O)$, the normal component on~$\G_\NN$ vanishes, i.e., 
\[
[\![q_h\cdot n_S]\!]|_S(x) = q_h(x) \cdot n_S = 0
\]
for all boundary sides $S\in \cS_h \cap \G_\NN$ and $x\in S$. This means that
$q_h \in \mathcal{BDM}(\O)$, if and only if $q_h \in \cL^1(\cT_h)^d$ and
\[
\int_{ \cup (\cS_h\setminus (\cS_h \cap \G_\DD))}  [\![q_h\cdot n_S]\!] r_h \dv{s} = 0 
\]
for all $r_h \in Z_h$.

\subsection{Nonlinear Laplace equation}

We consider the nonlinear Laplace problem with inhomogeneous Dirichlet data
on the L-shaped domain and let $\O=(-1,1)^2\setminus ([0,1]\times [-1,0])$,
$\G_\DD = \p\O$ and $g=0$, and define the Dirichlet data ~$u_\DD=u|_{\p\O}$
through restriction of the exact solution given in polar coordinates by
\[
u(r,\theta) = r^\d\sin(\d \theta)
\]
to the boundary. The choice of~$\d$ will be specified later in dependence of the choice of~$\s$.
The nonsmooth source term~$f$ is then given in polar coordinates by
\[
f(r,\theta) = -(2-\s)\d^{\s-1}(1-\d)r^{(\d-1)(\s-1)-1}\sin(\d\theta).
\]
We let $\d=(6/5)(1-1/\s)$. Then we have that $u \in W^{1,\s}(\O)$ but $u \notin W^{2,\s}(\O)$.
In what follows~$u_h \in X_h$ and~$p_h \in Y_h$ denote approximate solutions
to the primal and dual nonlinear Laplace problem obtained with the iterative scheme Variable-ADMM
(cf.~\cite{BarMil17}).

\begin{figure}[!htb]
\centering
\subfloat{\resizebox{6.6cm}{5cm}{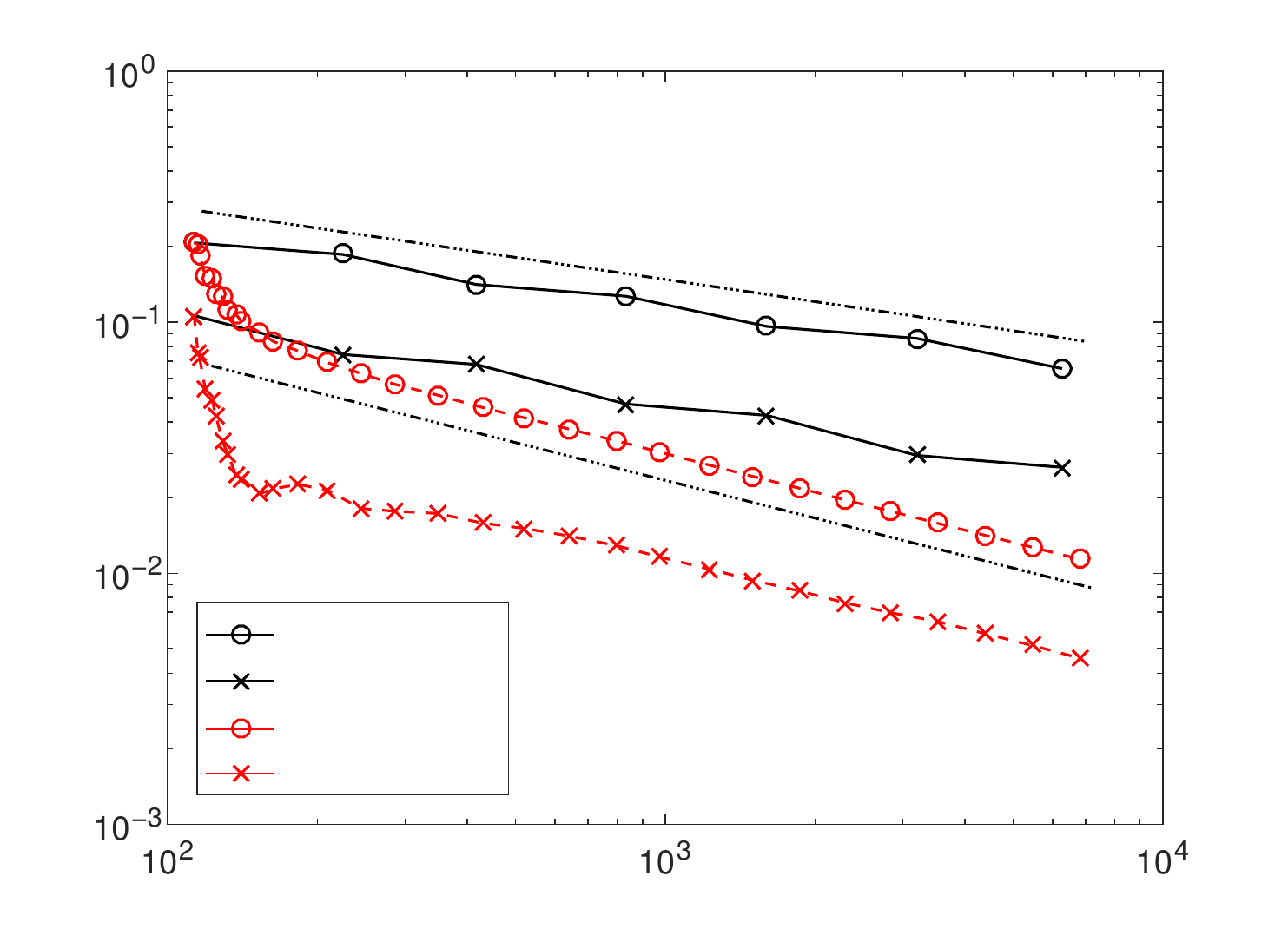}}
\subfloat{\resizebox{6.6cm}{5cm}{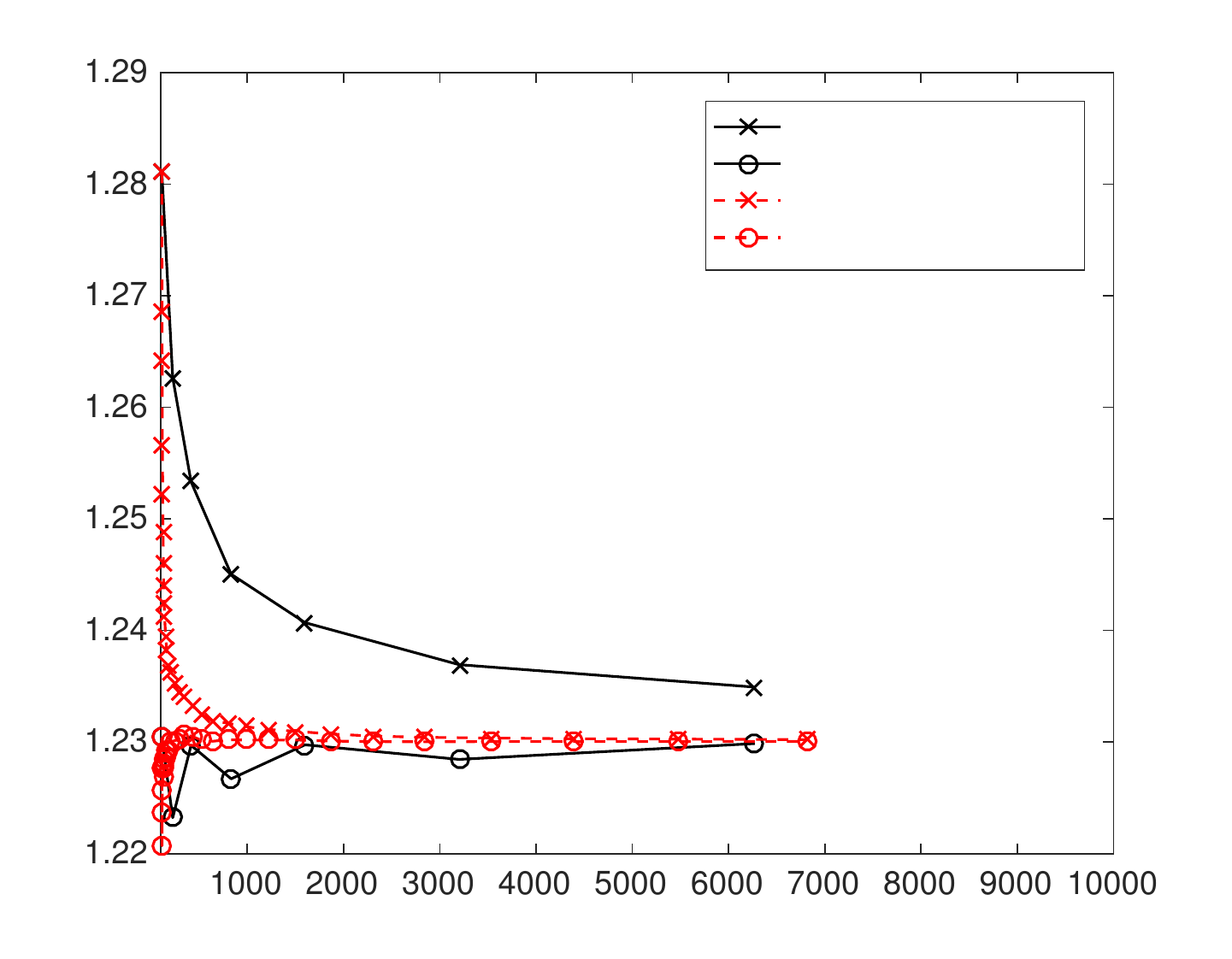}}\par
\subfloat{\resizebox{6.6cm}{5cm}{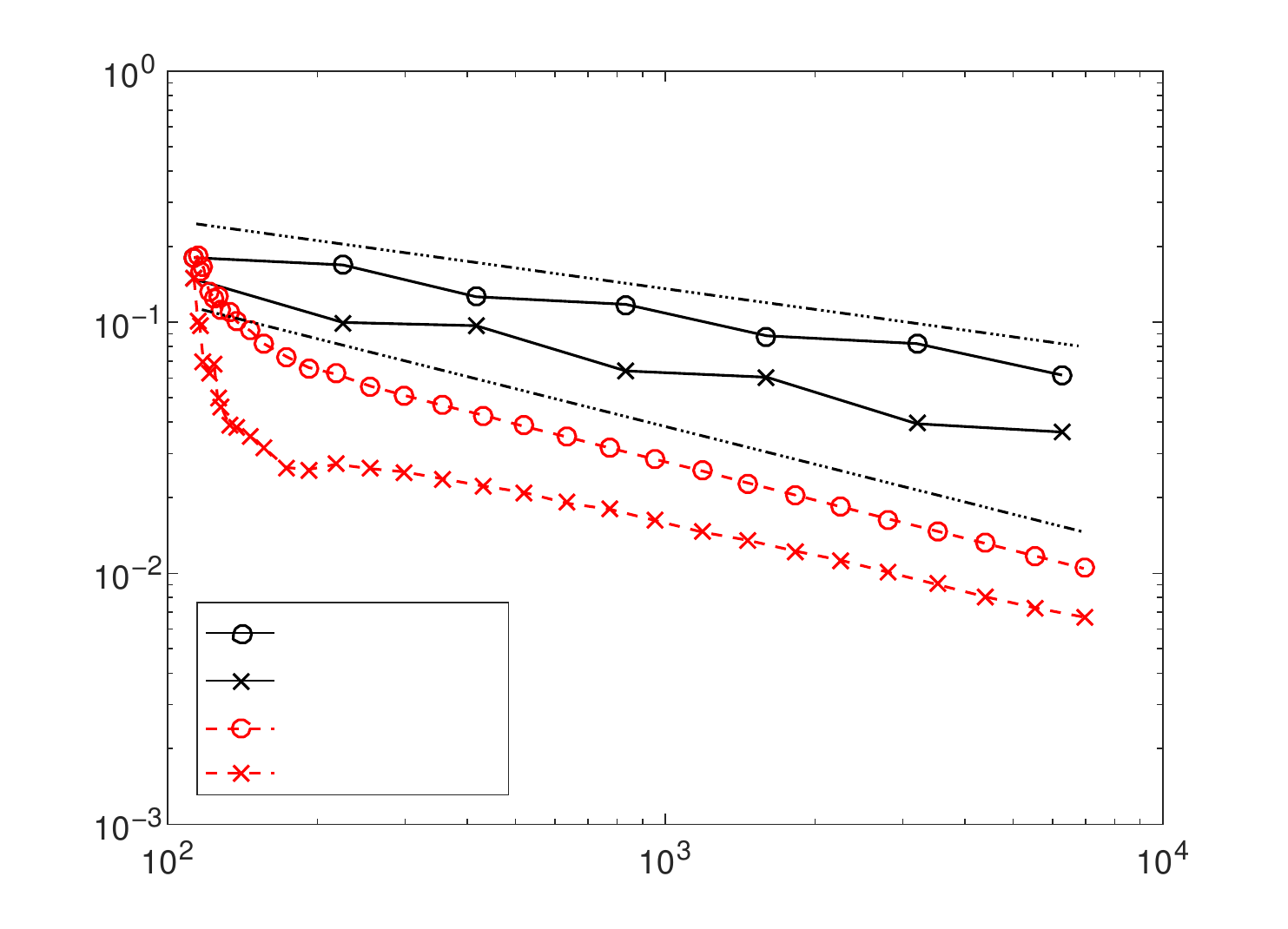}}
\subfloat{\resizebox{6.6cm}{5cm}{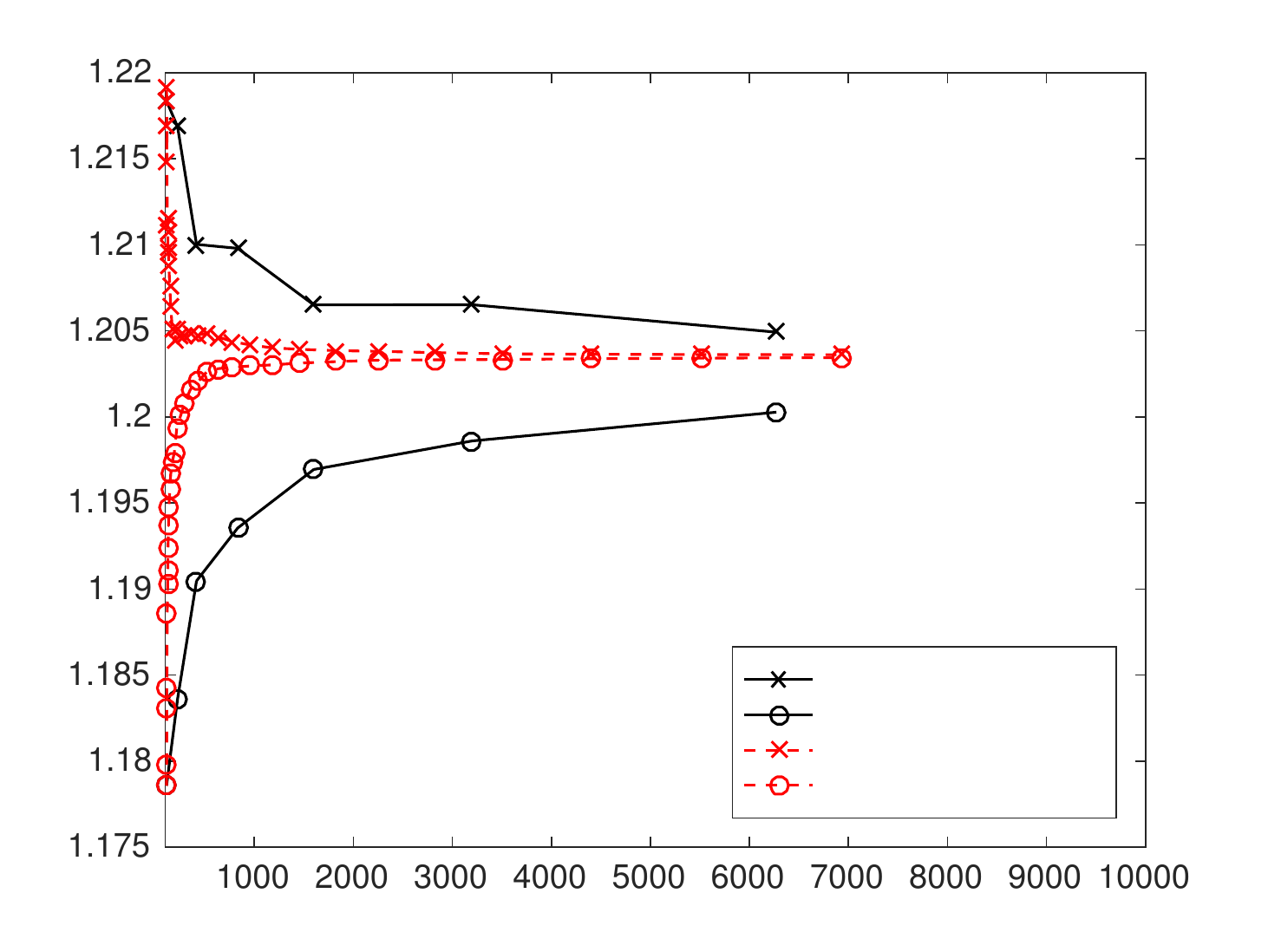}}
\caption{Primal-dual gap error estimators~$\heta_{\Delta_\s}^h$ and error $\vrho_{\Delta_\s}^{1/2}=\|V(\nabla u)-V(\nabla u_h)\|$ (left) and primal and dual energy~$E_{\Delta_\s}^h(u_h)$ and~$\hD_{\Delta_\s}^h(p_h)$ (right) for uniform and adaptive mesh refinement.
Top: Nonlinear Laplace problem with $\s=1.6$. Bottom: Nonlinear Laplace problem with $\s=1.2$.}\label{fig:error and energy pLaplace}
\end{figure}

In Figure~\ref{fig:error and energy pLaplace}
the error estimator~$\heta_{\Delta_\s}^h(u_h,p_h)$
and the error in the quasi-norm on the left-hand side of the estimate in
Proposition~\ref{prop:a posteriori pLaplace}
\[
\vrho_{\Delta_\s}^{1/2}=\|V(\nabla u)-V(\nabla u_h)\|
\]
are plotted against the number of degrees of freedom $N=|\cN_h|$ in a loglog-plot.
One can clearly observe that mesh adaptivity yields the quasi-optimal convergence rate
$\overline{h} \sim N^{-1/2}$. Particularly, the primal-dual gap error
estimator~$\heta_{\Delta_\s}^h(u_h,p_h)$ defines a reliable upper bound for the error in the quasi-norm.
On the right-hand side of Figure~\ref{fig:error and energy pLaplace} we displayed the energy
curves for the primal and dual energy $E_{\Delta_\s}^h(u_h)$ and~$\hD_{\Delta_\s}^h(p_h)$,
respectively. The primal and dual energy converge to the optimal value and the
primal-dual gap $E_{\Delta_\s}^h(u_h)-\hD_{\Delta_\s}^h(p_h)$ converges to zero as $N \to \infty$
and at a higher rate, when local mesh refinement is used. In Figure~\ref{fig:refined mesh} three
snapshots of the refined mesh are displayed, which show that the primal-dual gap error estimator
yields triangulations that are locally refined in the neighborhood of the singularity.
The high resolution is even more localized for $\s \to 1$, since the singularity at the reentrant
corner increases.

In Figure~\ref{fig:iteration numbers pLaplace} the iteration numbers for the Variable-ADMM
for the primal and dual problem are plotted versus the number of degrees of freedom for both
uniform and adaptive mesh refinement and for parameters $\s=1.6$ and $\s=1.2$.
The error tolerance for the residual in the Variable-ADMM was of order~$\mathcal{O}(h^2)$.
One can observe that the iteration numbers for the dual problem critically increase as~$\s$
is decreased.

\begin{figure}[!htb]
\captionsetup[subfigure]{labelformat=empty}
\centering
\subfloat{\includegraphics[width=3.5cm,height=3.5cm]{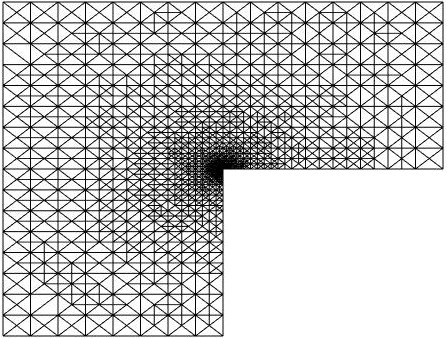}} \hspace*{5mm}
\subfloat{\includegraphics[width=3.5cm,height=3.5cm]{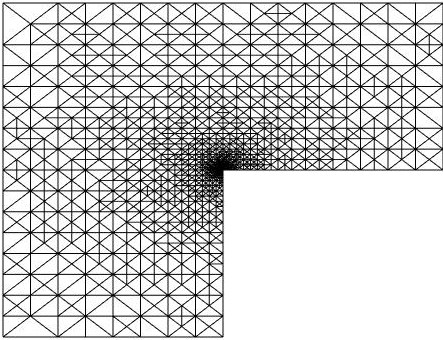}}\hspace*{5mm}
\subfloat{\includegraphics[width=3.5cm,height=3.5cm]{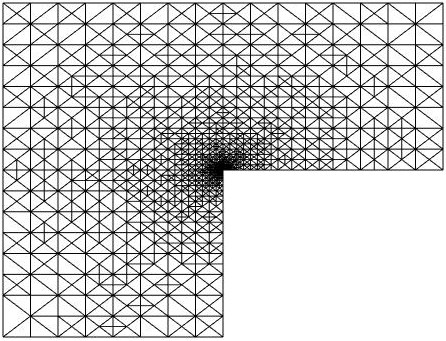}}
\caption{Snapshots of refined meshes for nonlinear Laplace problem with $\s=1.6$ (left), $\s=1.2$ (middle) and $\s=1.05$ (right). The mesh is locally refined in a neighborhood of the reentrant corner. The resolution at the reentrant corner increases as $\s \to 1$.}\label{fig:refined mesh}
\end{figure}

\begin{figure}[!htb]
\centering
\subfloat{\resizebox{6.6cm}{5cm}{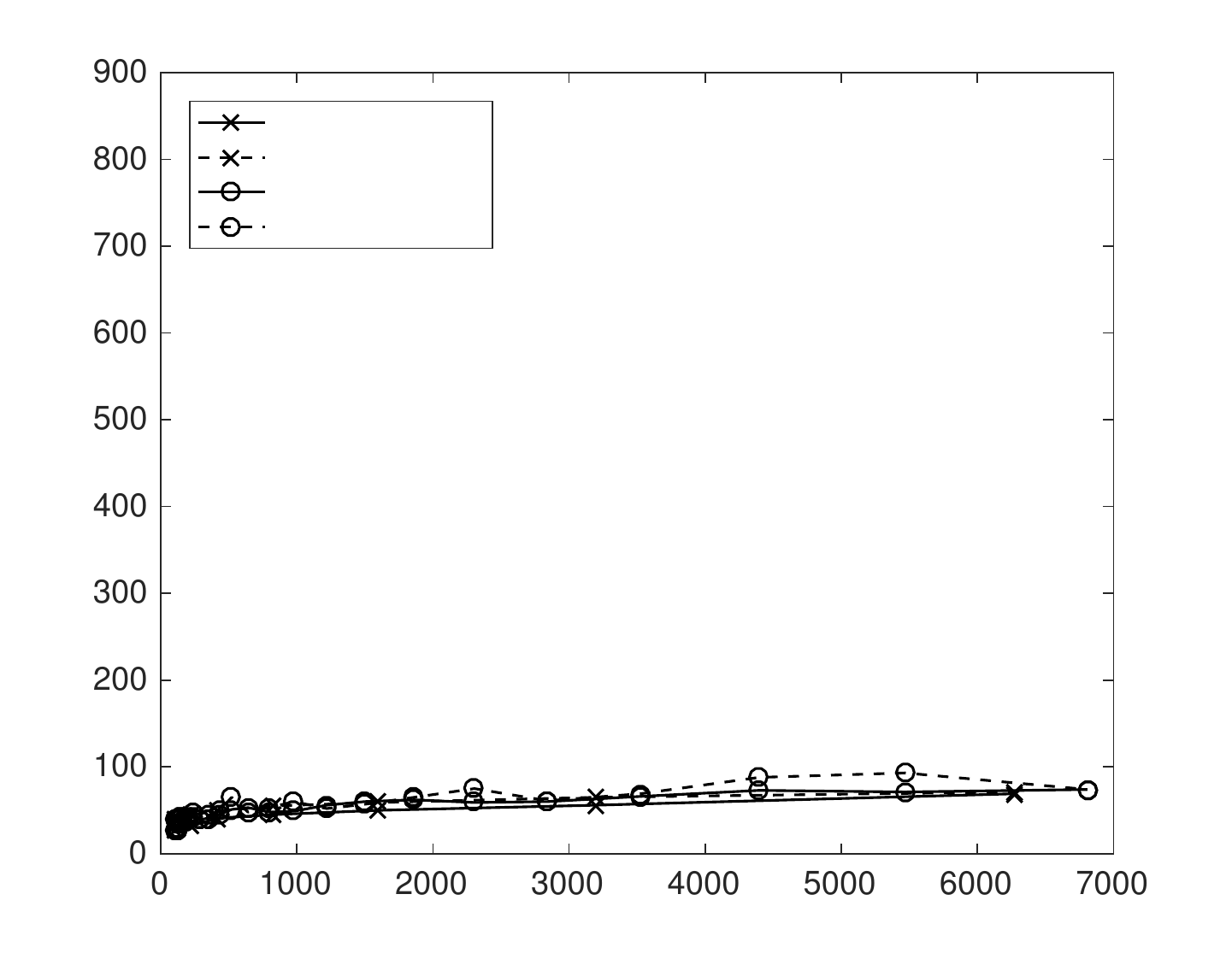}}
\subfloat{\resizebox{6.6cm}{5cm}{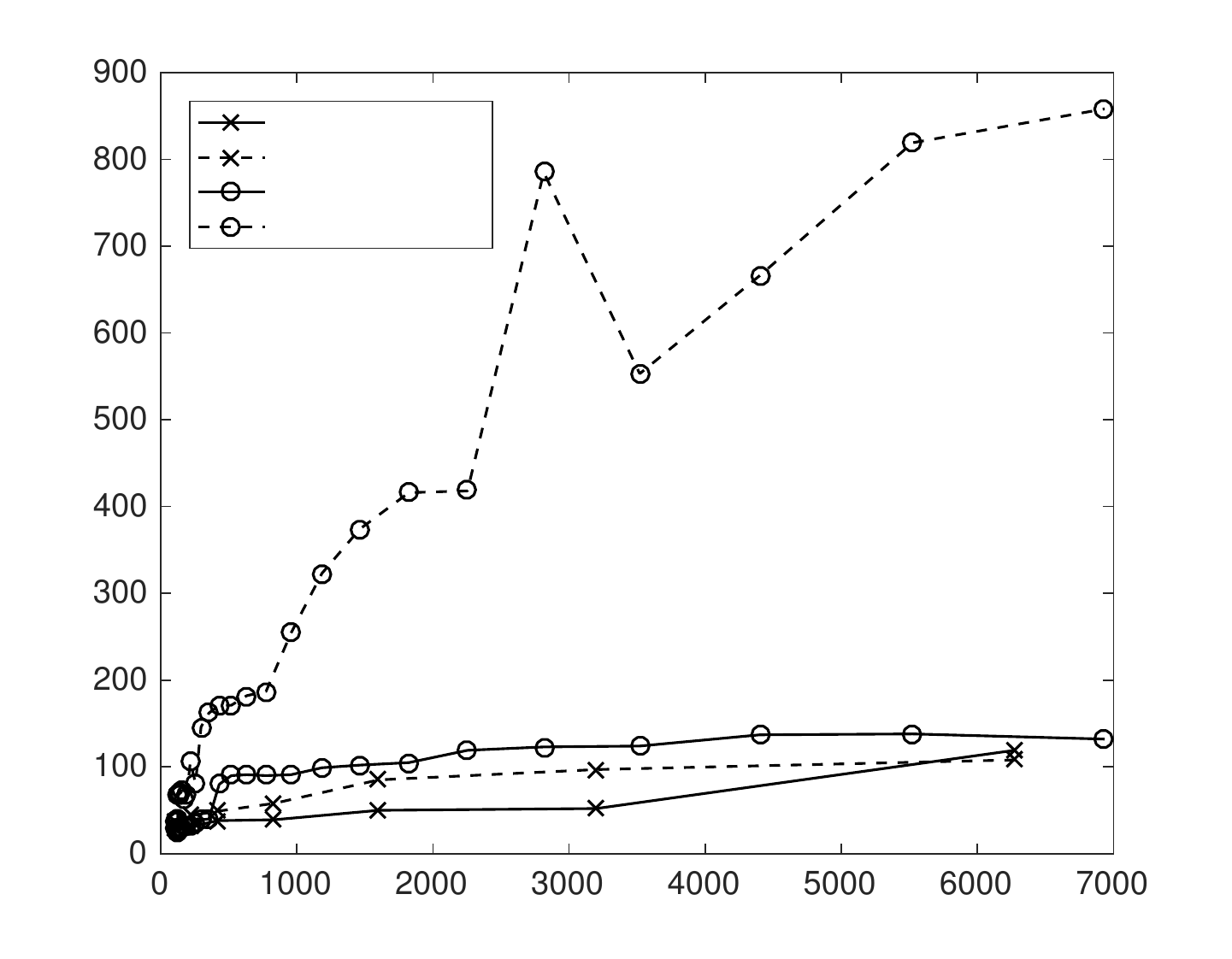}}
\caption{Iterations numbers for Variable-ADMM for the minimization of ~$E_{\Delta_\a}^h$
~$\hD_{\Delta_\a}^h$ for both uniform and adaptive refinement. Left: $\a=1.6$; right: $\a=1.2$.}\label{fig:iteration numbers pLaplace}
\end{figure}


Let us finally consider the residual-based error estimator
\[
\eta_{res}^h(u_h)^2 = \sum_{T \in \cT_h} \eta_{res}^{h,T}(u_h)^2
\]
from~\cite{LiuYan01,LiuYan01:2,LiuYan02,DieKre08,BelDieKre12} with
\[
\eta_{res}^{h,T}(u_h)^2 = \eta_E^{h,T}(u_h)^2 + \sum_{S \in \cS_h\setminus \p \O, S \subset \p T}\eta_J^{h,S}(u_h)^2
\]
and
\[\begin{split}
\eta_E^{h,T}(u_h)^2 &= \int_T (|\nabla u_h|^{\s-1} +h_T|f_h|)^{\s'-2}h_T^2|f_h|^2 \dv{x},\\
\eta_J^{h,S}(u_h)^2 &= \int_{\o_S} (|\nabla u_h| +|[\![\nabla u_h]\!]_S|)^{\s-2}|[\![\nabla u_h]\!]_S|^2 \dv{x},
\end{split}\]
where $\o_S = \bigcup\{T_1,T_2 \in \cT_h: \; S=T_1\cap T_2\}$ for
$S \in \cS_h \setminus \p \O$ and~$u_h$ is the unique discrete minimizer
of~$E_{\Delta_\s}^h$. The expression $[\![\nabla u_h]\!]_S$ denotes the jump
of~$\nabla u_h$ across an inner side $S \in \cS_h$ defined by
\[
[\![\nabla u_h]\!]_S = \nabla u_h|_{T_1} - \nabla u_h|_{T_2}
\]
for $S=T_1 \cap T_2$. The error estimator~$\eta_{res}^h(u_h)$ has been extensively studied in~\cite{LiuYan01,LiuYan01:2,LiuYan02,DieKre08,BelDieKre12}, where
the efficiency and reliability of the estimator has been proven and the linear
convergence as well as the optimality of the corresponding adaptive finite element
scheme have been shown.

\begin{figure}[!htb]
\centering
\subfloat{\resizebox{6.6cm}{5cm}{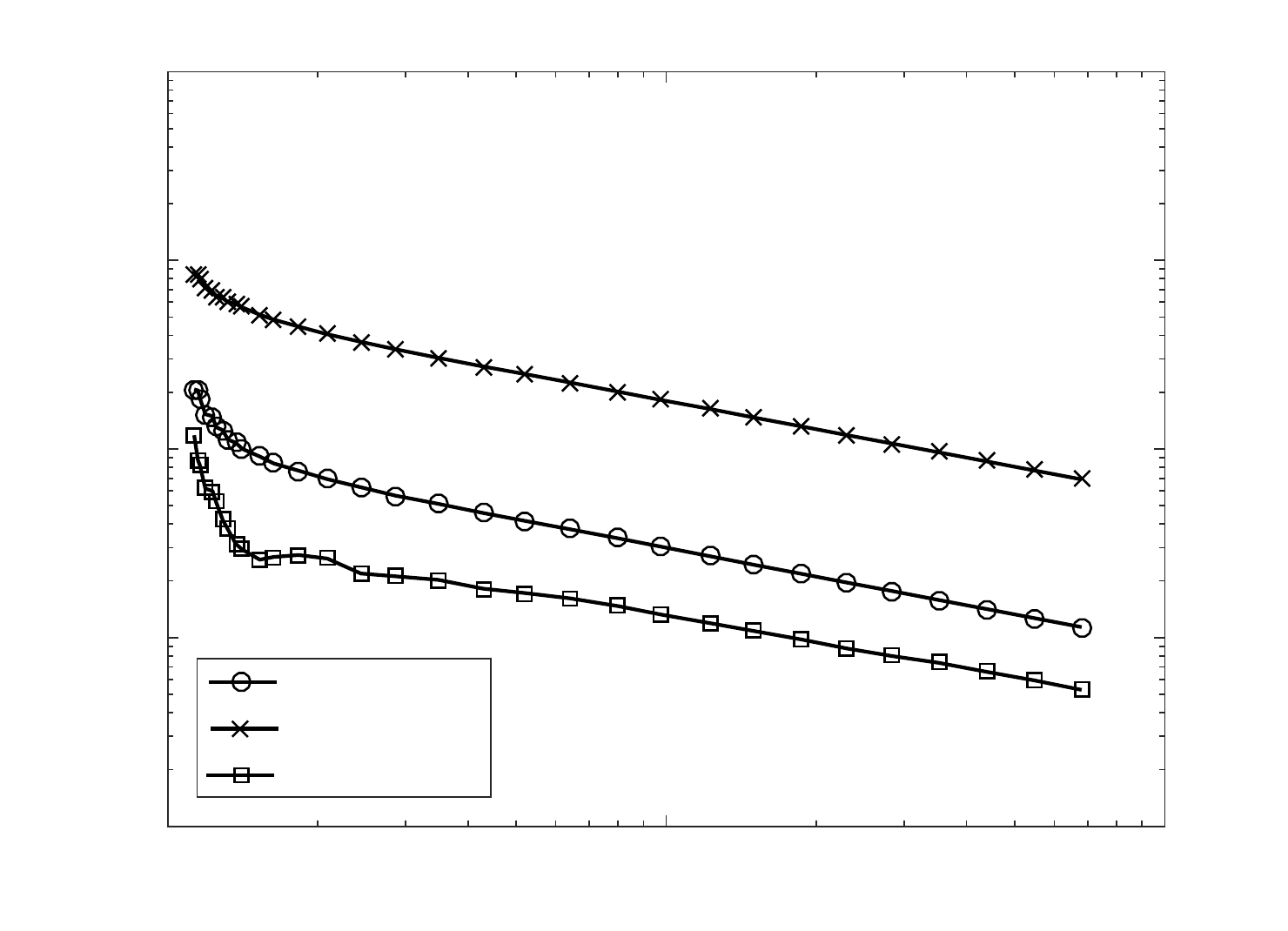}}
\subfloat{\resizebox{6.6cm}{5cm}{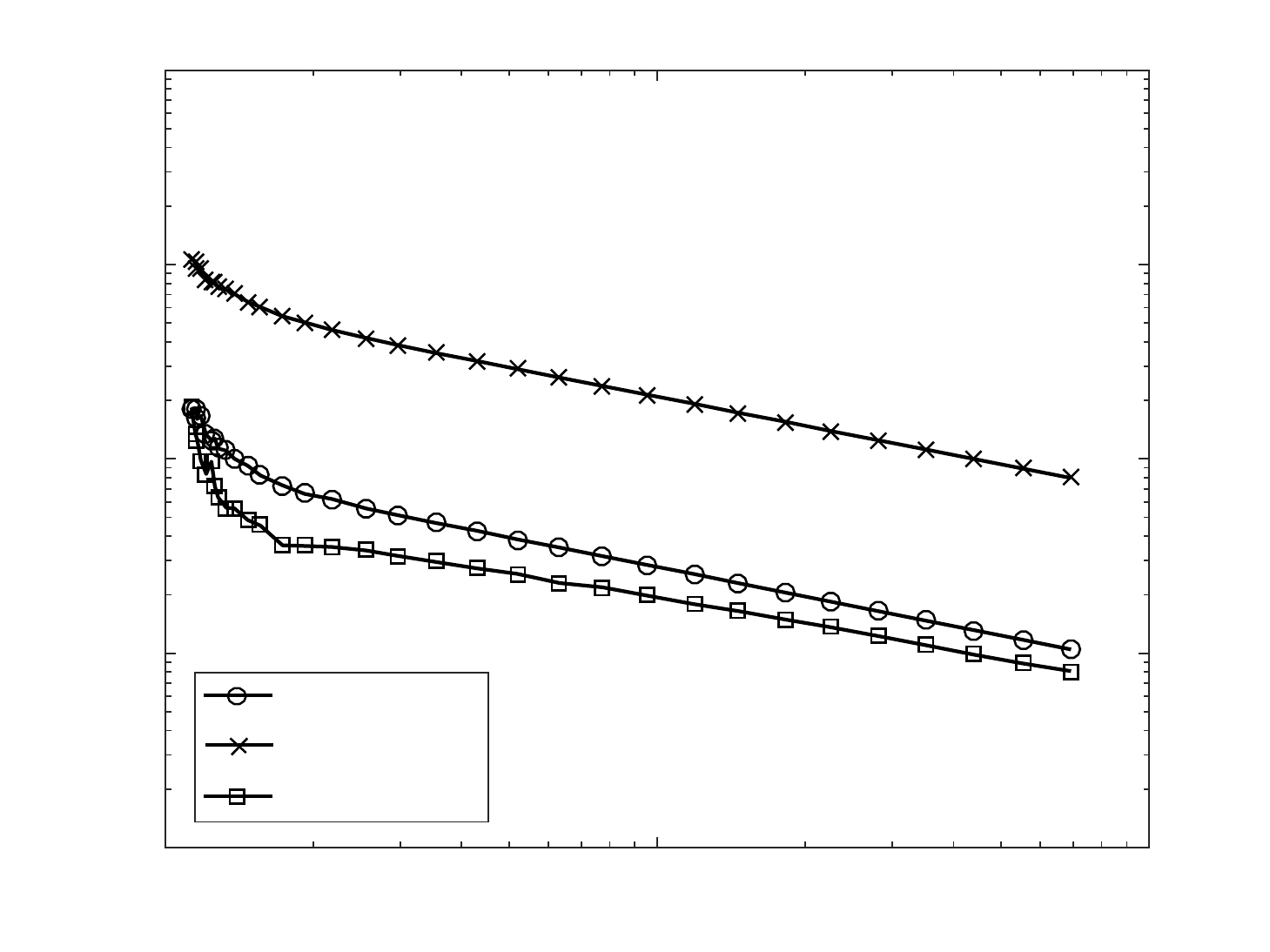}}
\caption{Primal-dual gap estimator~$\heta_{\Delta_\s}^h$, residual-based
estimator~$\eta_{res}^h$ and error $\vrho_{\Delta_\s}^{1/2}=\|V(\nabla u)-V(\nabla u_h)\|$ for a sequence of adaptively refined meshes driven by~$\heta_{\Delta_\s}^h$. Left: Nonlinear Laplace problem with $\s=1.6$. Right: Nonlinear Laplace problem with $\s=1.2$.}\label{fig:comp estimators}
\end{figure}

In Figure~\ref{fig:comp estimators} we compare the primal-dual gap error
estimator~$\heta_{\Delta_\s}^h(u_h,p_h)$ with the residual error estimator~$\eta_{res}^h(u_h)$
for the nonlinear Laplace problem with inhomogeneous Dirichlet data on the
L-shaped domain for $\s=1.6$ and $\s=1.2$ as before. One can observe that
both estimators decay at the same rate~$\mathcal{O}(N^{-1/2})$
on a sequence of locally refined meshes driven by an element marking
strategy based on~$\heta_{\Delta_\s}^h(u_h,p_h)$. However, the overestimation of the
primal-dual gap error estimator~$\heta_{\Delta_\s}^h(u_h,p_h)$ is moderate compared to
the residual-based error estimator~$\eta_{res}^h(u_h)$. While the overestimation
of~$\eta_{res}^h(u_h)$ for $\s=1.6$ and $\s=1.2$ do not differ significantly, the gap between
the primal-dual gap error estimator and the error diminishes for $\s=1.2$. Let us also remark that
in the proofs of the reliability and the efficiency of the residual-based error
estimator~$\eta_{res}^h(u_h)$ it is crucial that~$u_h$ is the unique solution to the
primal nonlinear Laplace problem in~$X_h$, cf.~\cite{DieKre08}. Its robustness regarding
inexact iterative solutions is not addressed in the aforementioned articles.

\subsection{Rudin-Osher-Fatemi image denoising}

We let $\O=(-1,1)^2$ and consider two examples, the first one with homogeneous Neumann
boundary conditions and the second one with homogeneous Dirichlet boundary conditions,
for which we have an explicit solution at hand. In the case of Dirichlet boundary conditions
the dual energy functional~$D_{\rm rof}$ is maximized over~$H(\diver;\O)$ instead
of~$H_\NN(\diver;\O)$. The calculations remain valid, but in general it is nontrivial
to guarantee the existence of solutions for Dirichlet boundary conditions.

\begin{example}\label{ex:square}
We set $\G_\DD=\emptyset$, $\G_\NN=\p\O$, $\a=100$, and $g=\chi_{B_{1/2}^\infty(0)}$ the
characteristic function of $B_{1/2}^\infty(0)=\{(x_1,x_2) \in \R^2: \; \max\{|x_1|,|x_2|\}\leq 1/2\}$.
\end{example}

\begin{figure}[!htb]
\centering
\subfloat{\resizebox{6.6cm}{5cm}{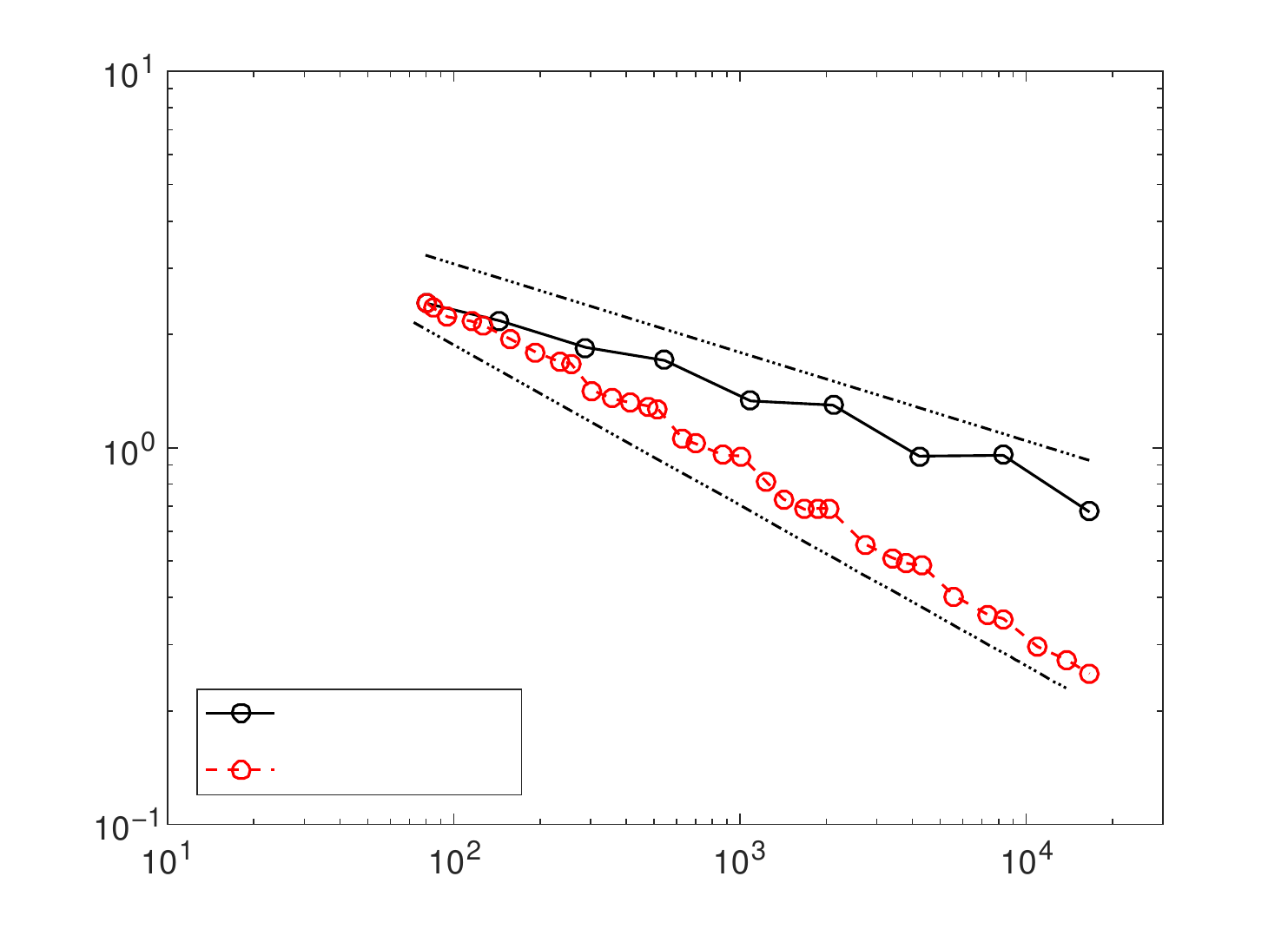}}
\subfloat{\resizebox{6.6cm}{5cm}{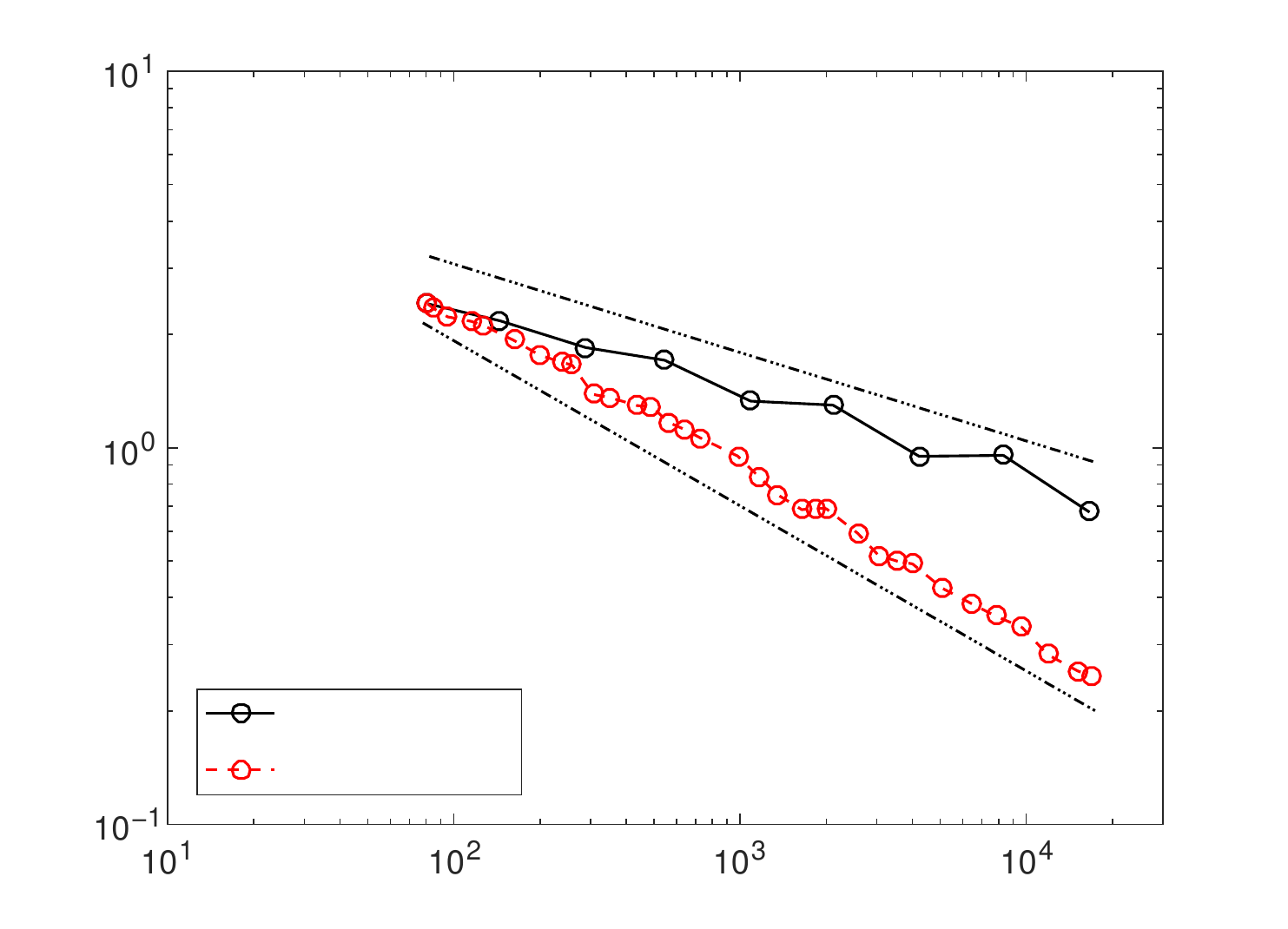}}
\caption{Error estimator~$\eta_{\rm rof}^h$ for Example~\ref{ex:square} with discretization of the
dual problem with continuous finite element space~$Y_h^C$ (left) and $H(\diver;\O)$-conforming
finite element space~$Y_h^{dC}$ (right) for uniform and adaptive mesh refinement.}\label{fig:error ROF square}
\end{figure}

In Figure~\ref{fig:error ROF square} the error estimator~$\eta_{\rm rof}^h$
is plotted against the number of degrees of freedom $N=|\cN_h|$ using a logarithmic scaling on both axes both
for uniform and adaptive mesh refinement and with the dual problem discretized with the
continuous finite element space $Y_h^C=\cS^1(\cT_h)^d$ and the $H(\diver;\O)$-conforming
finite element space $Y_h^{dC}=\cL^1(\cT_h)^d \cap H_\NN(\diver;\O)$.
Again, one can observe that using locally refined meshes with~$Y_h^C$ as the
discrete space for the dual problem yields a better convergence
rate $\overline{h}^{0.76} \sim N^{-0.38}$ as compared to
uniform refinement with an experimental convergence rate of~$\overline{h}^{0.47}$.
For the choice~$Y_h^{dC}$ we record the rates $\overline{h}^{0.81} \sim N^{-0.4}$ (adaptive)
and $\overline{h}^{0.47} \sim N^{-0.24}$ (uniform).
The choice of the finite element space for the discretization of the dual problem does not significantly
affect the rate of convergence of the primal-dual gap error estimator~$\eta_{\rm rof}^h$.

\begin{example}\label{ex:circle}
We set $\G_\DD=\p\O$, $\G_\NN=\emptyset$, $\a=10$ and $g=\chi_{B_{1/2}^2(0)}$ with
$B_{1/2}^2(0)=\{x \in \R^2: \; |x|\leq 1/2\}$.
\end{example}

In this case the exact solution is given by $u = (3/5)\chi_{B_{1/2}^2(0)}$,
cf.~\cite{Bar15}.

\begin{figure}[!htb]
\centering
\subfloat{\resizebox{6.6cm}{5.0cm}{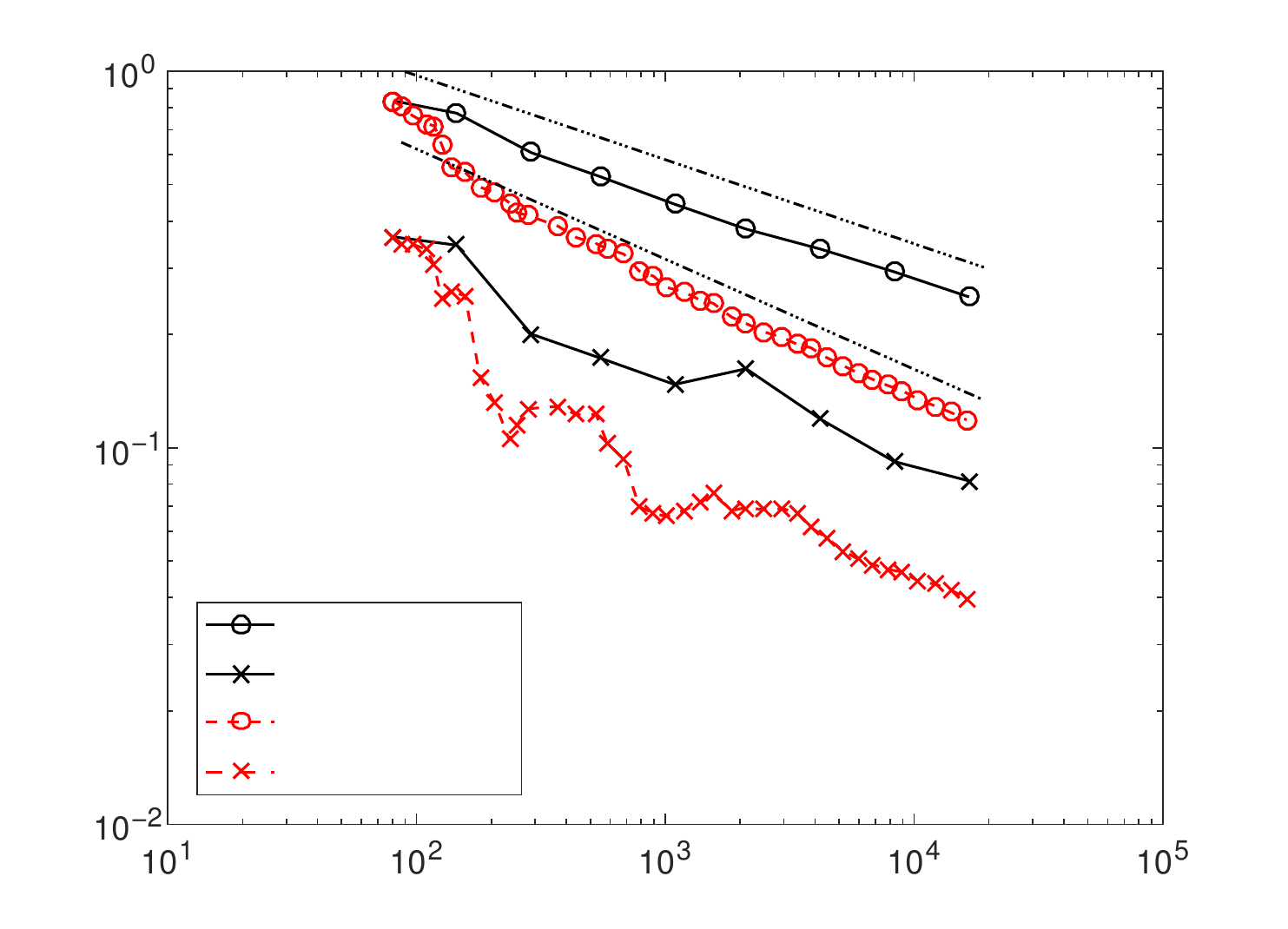}}
\subfloat{\resizebox{6.6cm}{5.0cm}{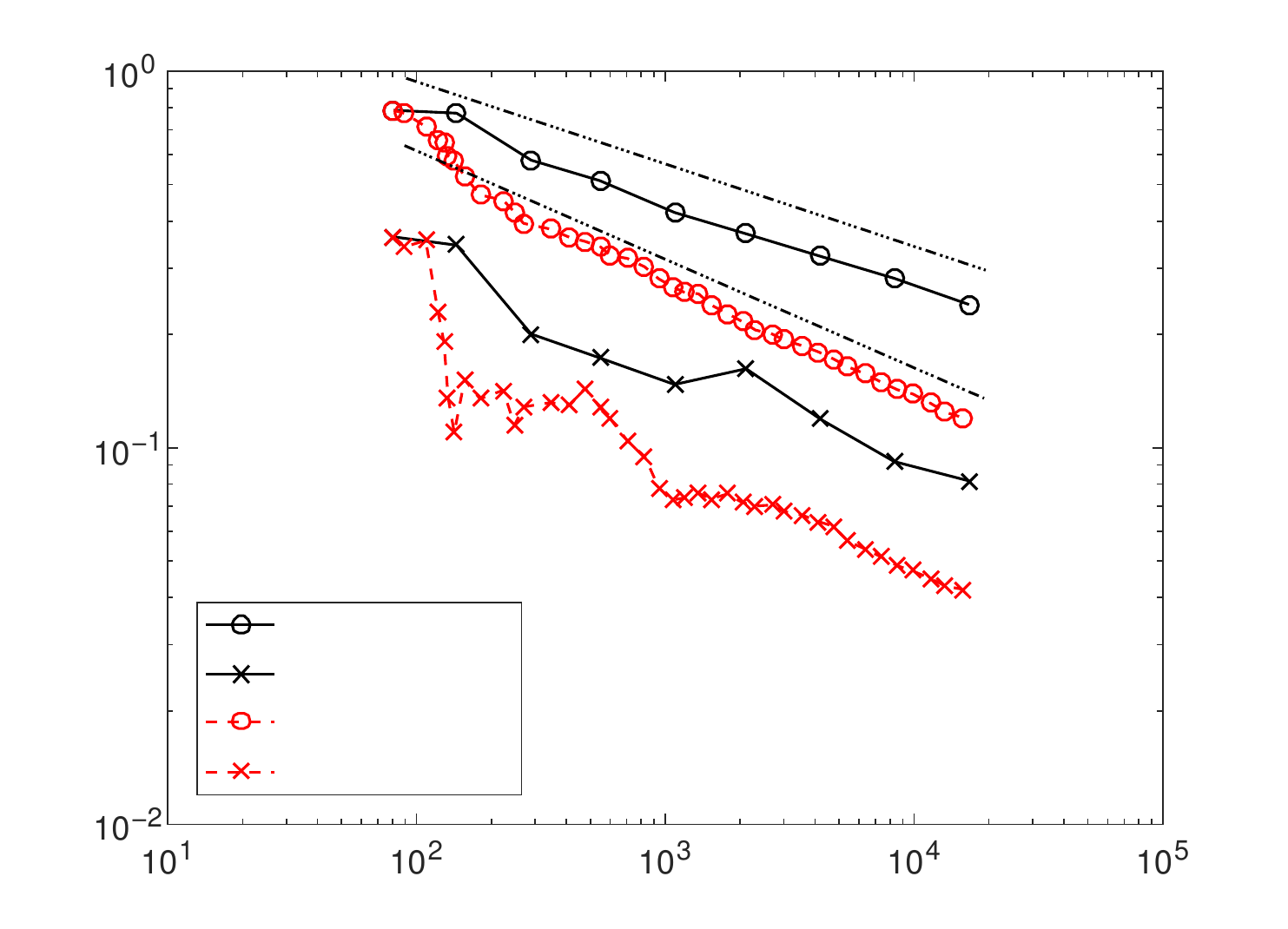}}
\caption{Primal-dual gap error estimator~$\eta_{\rm rof}^h$ and $L^2$-error $\vrho_{\rm rof}^{1/2}=(\a/2)^{1/2}\|u-u_h\|$ for Example~\ref{ex:circle} with discretization of the
dual problem with continuous finite element space~$Y_h^C$ (left) and $H(\diver;\O)$-conforming
finite element space~$Y_h^{dC}$ (right) for uniform and adaptive mesh refinement.}\label{fig:error ROF circle}
\end{figure}

\begin{figure}[!htb]
\centering
\subfloat{\resizebox{6.6cm}{5cm}{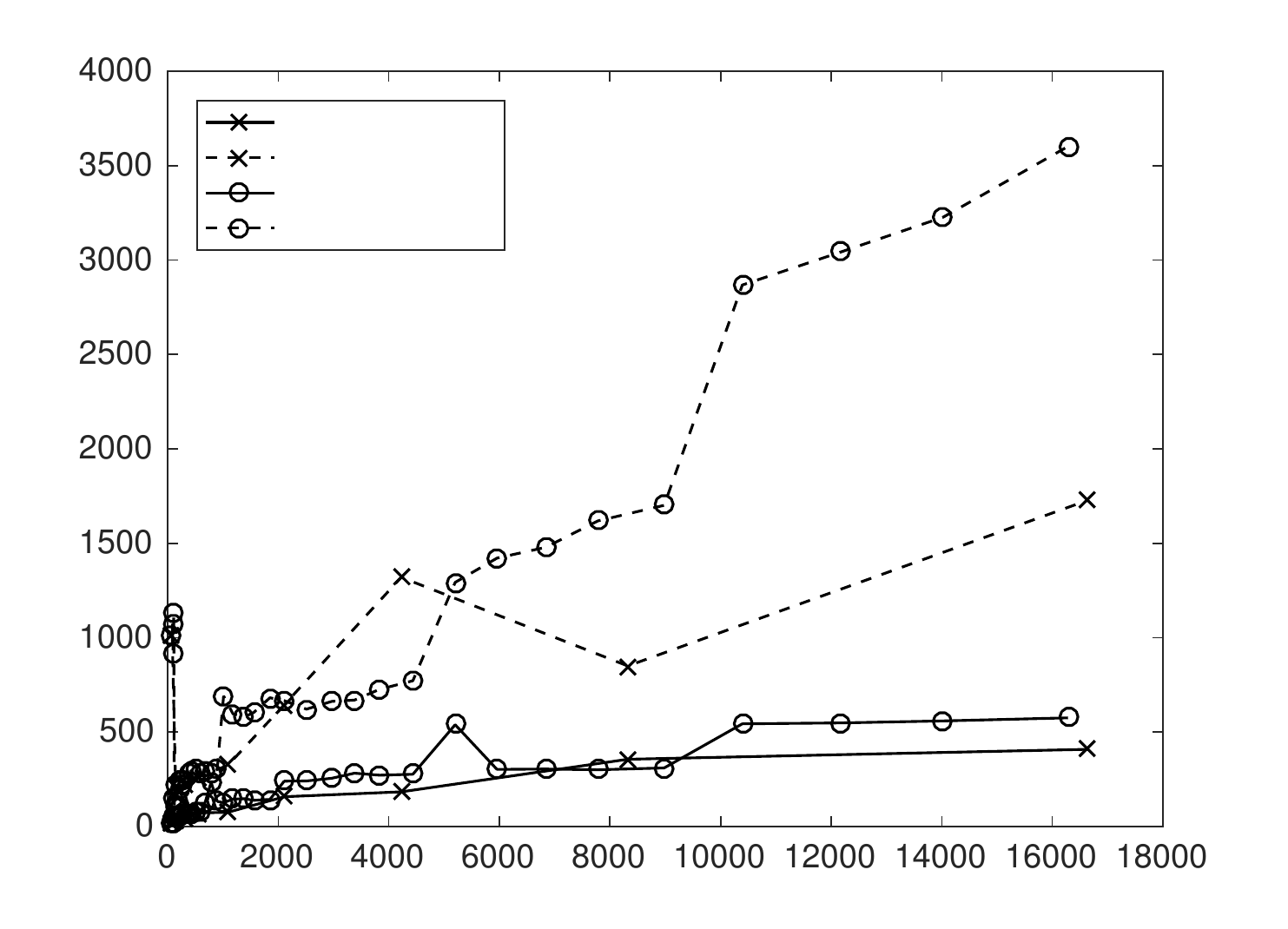}}
\subfloat{\resizebox{6.6cm}{5cm}{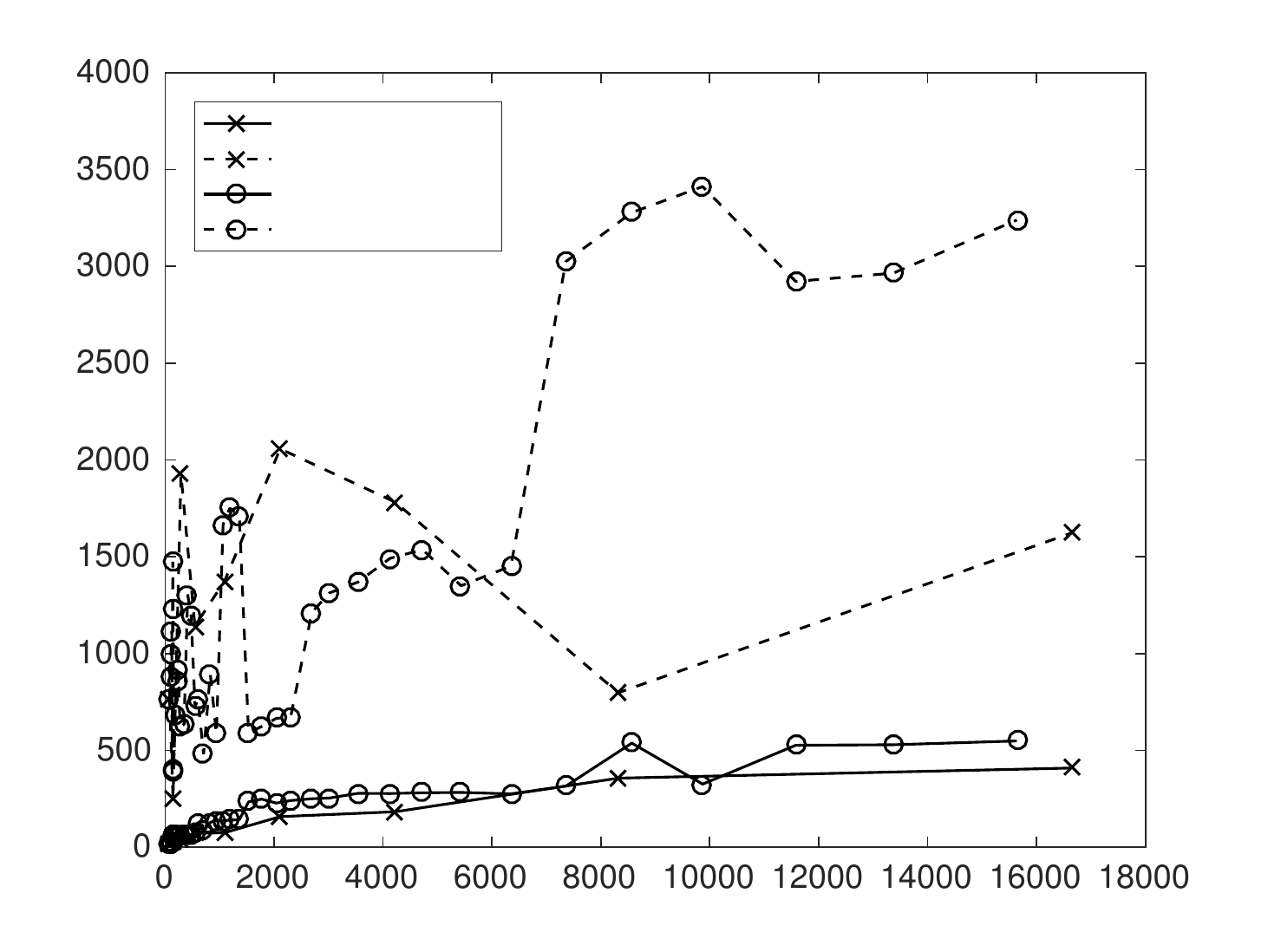}}
\caption{Iterations numbers for Variable-ADMM for the minimization of~$E_{\rm rof}^h$
and~$-D_{\rm rof}^h$ for both uniform and adaptive refinement. Left: $Y_h=Y_h^C$; right: $Y_h=Y_h^{dC}$.}\label{fig:iteration numbers ROF circle}
\end{figure}

In Figure~\ref{fig:error ROF circle} the error
estimator~$\eta_{\rm rof}^h$ and the $L^2$-error
\[
\vrho_{\rm rof}^{1/2} = (\a/2)^{1/2} \|u-u_h\|
\]
are plotted against the number of degrees of freedom in a loglog-plot and again, as before,
both for uniform and adaptive mesh refinement and for the discretization of the dual problem
with~$Y_h^C$ (left) and~$Y_h^{dC}$ (right). The plot underlines
that the quantity~$\eta_{\rm rof}^h$ defines a reliable
estimator for the $L^2$-error~$\vrho_{\rm rof}^{1/2}$ as predicted by
Proposition~\ref{prop:a posteriori ROF}. One can, once again, observe that adaptive mesh refinement
leads to an improvement of the convergence rate from~$\overline{h}^{0.44} \sim N^{-0.22}$
to~$\overline{h}^{0.62} \sim N^{-0.31}$ for both discretization methods for the dual problem.
In Figure~\ref{fig:iteration numbers ROF circle} the iteration numbers for the Variable-ADMM
for the primal and dual problem are plotted against the number of degrees of freedom for both
uniform and adaptive mesh refinement and for discretizations of the dual problem with
$Y_h=Y_h^C$ and $Y_h=Y_h^{dC}$.
The error tolerance for the residual in the Variable-ADMM was of order~$\mathcal{O}(h)$.
The iteration numbers for $Y_h=Y_h^C$ and $Y_h=Y_h^{dC}$ do not differ significantly.
However, one can observe that the iteration numbers of the Variable-ADMM as a function
of the degrees of freedom grow significantly faster for the dual problem compared to
the primal problem reflecting the weaker coercivity property.

\begin{figure}[!htb]
\captionsetup[subfigure]{labelformat=empty}
\centering
\subfloat{\includegraphics[width=3cm,height=3cm]{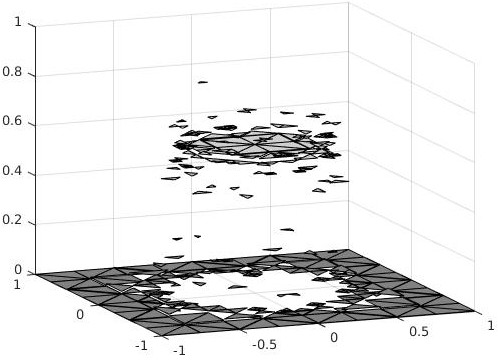}} \hspace*{5mm}
\subfloat{\includegraphics[width=3cm,height=3cm]{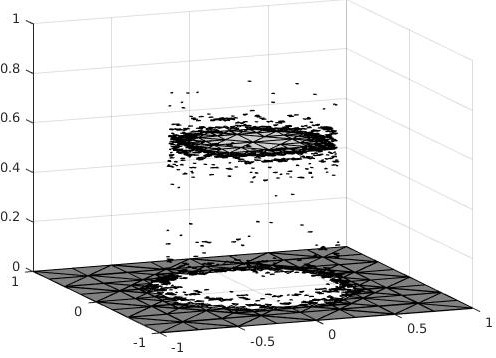}}\hspace*{5mm}
\subfloat{\includegraphics[width=3cm,height=3cm]{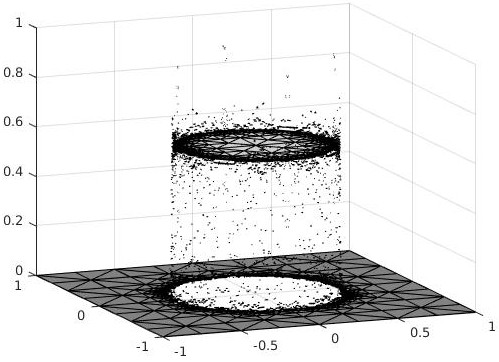}}\par
\subfloat{\includegraphics[width=3cm,height=3cm]{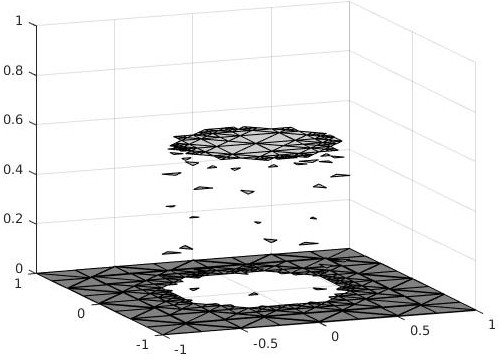}}\hspace*{5mm}
\subfloat{\includegraphics[width=3cm,height=3cm]{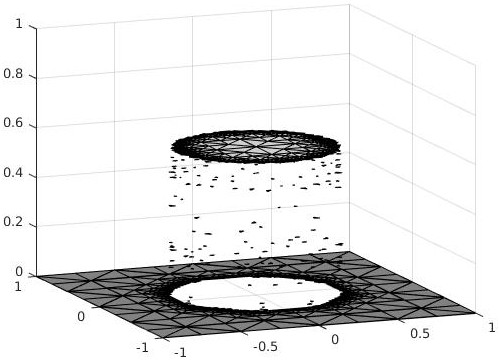}}\hspace*{5mm}
\subfloat{\includegraphics[width=3cm,height=3cm]{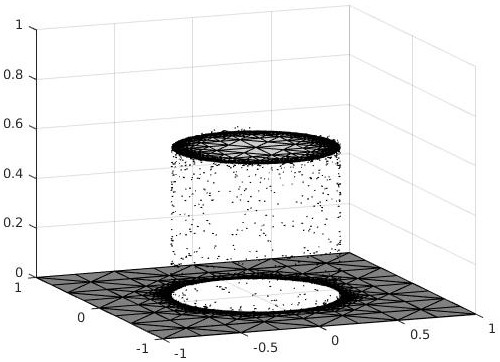}}
\caption{Piecewise constant approximations $\overline{u}_h=(1/\a)\diver p_h+g_h$ for a sequence of adaptively refined triangulations for Example~\ref{ex:circle}.
Top: Dual variable is approximated in $Y_h^C=\cS^1(\cT_h)^d$. Bottom: Dual variable is approximated in $Y_h^{dC}=\cL^1(\cT_h)^d\cap H_\NN(\diver;\O)$. One can observe oscillations of~$\overline{u}_h$ along the jump set for the discretization of the dual ROF problem with~$Y_h^C$.}\label{fig:up0 circle}
\end{figure}

In Figure~\ref{fig:up0 circle} we depicted for a sequence of adaptively refined triangulations
the piecewise constant approximations $\overline{u}_h = (1/\a)\diver p_h + g_h$ with
$p_h \in Y_h^C$ (top) and $p_h \in Y_h^{dC}$ (bottom), cf. Proposition~\ref{prop:piecewise constant approximation}.
Although the different discretization
methods for the dual problem do not affect the convergence rates in the presented experiments,
the discretization of the
dual problem with the continuous finite element space~$Y_h^C$ causes oscillations
in~$\overline{u}_h$ along the jump set.

\section{Conclusion}

We have seen that the primal-dual gap error estimator defines a
reliable upper bound with constant one for the error in the energy
for convex minimization problems. For uniformly convex minimization problems
it also controls the error with respect to a distance induced by the uniform convexity.
The primal-dual gap error estimator has been introduced in~\cite{Rep00} in an
abstract setting and has been applied to several minimization problems in an
infinite-dimensional framework. We extended the theory to general
finite element discretizations of convex minimization problems and applied
the theory to the nonlinear Laplace problem and the ROF problem, which serve
as model problems for a wide class of convex minimization problems. The theoretical
results, especially the reliability of the primal-dual gap error estimator, has been
confirmed in several numerical experiments. In order to compute the estimator
we approximately solved the primal and dual problems using the Variable-ADMM
provided in~\cite{BarMil17}. Yet, it seems necessary to consider
more efficient strategies to construct feasible functions especially for the dual problems.

\bibliographystyle{plain}
\bibliography{dual_conf}

\end{document}

%% file: includes/sb_macros.tex
\def\R{\mathbb{R}}

\def\cI{\mathcal{I}}

\def\cL{\mathcal{L}}

\def\cN{\mathcal{N}}

\def\cS{\mathcal{S}}
\def\cT{\mathcal{T}}

\def\a{\alpha}
\def\b{\beta}
\def\g{\gamma}
\def\d{\delta}

\def\l{\lambda}
\def\s{\sigma}
\def\p{\partial}
\def\o{\omega}
\def\veps{\varepsilon}
\def\vrho{\varrho}

\def\O{\Omega}

\def\G{\Gamma}
\def\GD{{\Gamma_{\rm D}}}
\def\GN{{\Gamma_{\rm N}}}

\def\DD{{\rm D}} 
\def\NN{{\rm N}} 
 
\def\wto{\rightharpoonup}

\def\tu{\widetilde{u}}

\newcommand{\dv}[1]{\,{\mathrm d}#1}

\newcommand{\wcheck}[1]{#1\hspace{-.8ex}\mbox{\huge {\lower.45ex \hbox{$\textstyle \check{}$}}} \hspace{.5ex}}

\DeclareMathOperator{\diver}{div}

\let\oldmarginpar\marginpar
\renewcommand\marginpar[1]{
  \oldmarginpar[\raggedleft\footnotesize #1]
  {\raggedright\footnotesize #1}}
\usepackage{ifthen}
\usepackage[normalem]{ulem}

\newtheorem{definition}{Definition}
\newtheorem{lemma}[definition]{Lemma}
\newtheorem{proposition}[definition]{Proposition}
\newtheorem{theorem}[definition]{Theorem}

\newtheorem{remark}[definition]{Remark}
\newtheorem{remarks}[definition]{Remarks}
\newtheorem{example}[definition]{Example}

\numberwithin{definition}{section}

\definecolor{tourquoise}{RGB}{0,170,180}	
\definecolor{darkred}{RGB}{238,34,34}		
\definecolor{darkgreen}{RGB}{0,190,0}		
\definecolor{lightgray}{RGB}{210,210,210}	
\definecolor{deepblue}{RGB}{0,0,240}		
\definecolor{darkgray}{RGB}{144,144,144}	
\definecolor{kingblue}{RGB}{64,96,224}		
\definecolor{gold}{RGB}{240,208,0}		
\definecolor{verydarkred}{RGB}{176,0,0}		

%% file: pLaplace_Lshape_1pt6_error.pdf_t
\begin{picture}(0,0)%
\includegraphics{pLaplace_Lshape_1pt6_error.pdf}%
\end{picture}%
\setlength{\unitlength}{3947sp}%
\begingroup\makeatletter\ifx\SetFigFont\undefined%
\gdef\SetFigFont#1#2#3#4#5{%
  \reset@font\fontsize{#1}{#2pt}%
  \fontfamily{#3}\fontseries{#4}\fontshape{#5}%
  \selectfont}%
\fi\endgroup%
\begin{picture}(7102,5127)(0,-12362)
\put(2281,-12286){\makebox(0,0)[lb]{\smash{{\SetFigFont{12}{14.4}{\familydefault}{\mddefault}{\updefault}{\color[rgb]{0,0,0}Number of degrees of freedom ($N$)}%
}}}}
\put(397,-9901){\rotatebox{90.0}{\makebox(0,0)[lb]{\smash{{\SetFigFont{12}{14.4}{\familydefault}{\mddefault}{\updefault}{\color[rgb]{0,0,0}Error}%
}}}}}
\put(5551,-10636){\makebox(0,0)[lb]{\smash{{\SetFigFont{10}{12.0}{\familydefault}{\mddefault}{\updefault}{\color[rgb]{0,0,0}$cN^{-0.5}$}%
}}}}
\put(5176,-8911){\makebox(0,0)[lb]{\smash{{\SetFigFont{10}{12.0}{\familydefault}{\mddefault}{\updefault}{\color[rgb]{0,0,0}$cN^{-0.29}$}%
}}}}
\put(1576,-11041){\makebox(0,0)[lb]{\smash{{\SetFigFont{10}{12.0}{\familydefault}{\mddefault}{\updefault}{\color[rgb]{0,0,0}$\vrho_{\Delta_\s}^{1/2}$ (uniform)}%
}}}}
\put(1576,-11557){\makebox(0,0)[lb]{\smash{{\SetFigFont{10}{12.0}{\familydefault}{\mddefault}{\updefault}{\color[rgb]{0,0,0}$\vrho_{\Delta_\s}^{1/2}$ (adaptive)}%
}}}}
\put(1576,-10765){\makebox(0,0)[lb]{\smash{{\SetFigFont{10}{12.0}{\familydefault}{\mddefault}{\updefault}{\color[rgb]{0,0,0}$\heta_{\Delta_\s}^h$ (uniform)}%
}}}}
\put(1576,-11305){\makebox(0,0)[lb]{\smash{{\SetFigFont{10}{12.0}{\familydefault}{\mddefault}{\updefault}{\color[rgb]{0,0,0}$\heta_{\Delta_\s}^h$ (adaptive)}%
}}}}
\end{picture}%

%% file: pLaplace_Lshape_1pt6_energy_vers2.pdf_t
\begin{picture}(0,0)%
\includegraphics{pLaplace_Lshape_1pt6_energy_vers2.pdf}%
\end{picture}%
\setlength{\unitlength}{3947sp}%
\begingroup\makeatletter\ifx\SetFigFont\undefined%
\gdef\SetFigFont#1#2#3#4#5{%
  \reset@font\fontsize{#1}{#2pt}%
  \fontfamily{#3}\fontseries{#4}\fontshape{#5}%
  \selectfont}%
\fi\endgroup%
\begin{picture}(6902,5377)(0,-12362)
\put(2401,-12286){\makebox(0,0)[lb]{\smash{{\SetFigFont{12}{14.4}{\familydefault}{\mddefault}{\updefault}{\color[rgb]{0,0,0}Number of degrees of freedom}%
}}}}
\put(301,-10441){\rotatebox{90.0}{\makebox(0,0)[lb]{\smash{{\SetFigFont{12}{14.4}{\familydefault}{\mddefault}{\updefault}{\color[rgb]{0,0,0}Primal and dual energy}%
}}}}}
\put(4405,-7741){\makebox(0,0)[lb]{\smash{{\SetFigFont{12}{14.4}{\familydefault}{\mddefault}{\updefault}{\color[rgb]{0,0,0}$E_{\Delta_\s}^h(u_h)$ (uniform)}%
}}}}
\put(4405,-8437){\makebox(0,0)[lb]{\smash{{\SetFigFont{12}{14.4}{\familydefault}{\mddefault}{\updefault}{\color[rgb]{0,0,0}$\hD_{\Delta_\s}^h(p_h)$ (adaptive)}%
}}}}
\put(4405,-8197){\makebox(0,0)[lb]{\smash{{\SetFigFont{12}{14.4}{\familydefault}{\mddefault}{\updefault}{\color[rgb]{0,0,0}$E_{\Delta_\s}^h(u_h)$ (adaptive)}%
}}}}
\put(4405,-7981){\makebox(0,0)[lb]{\smash{{\SetFigFont{12}{14.4}{\familydefault}{\mddefault}{\updefault}{\color[rgb]{0,0,0}$\hD_{\Delta_\s}^h(p_h)$ (uniform)}%
}}}}
\end{picture}%

%% file: pLaplace_Lshape_1pt2_error.pdf_t
\begin{picture}(0,0)%
\includegraphics{pLaplace_Lshape_1pt2_error.pdf}%
\end{picture}%
\setlength{\unitlength}{3947sp}%
\begingroup\makeatletter\ifx\SetFigFont\undefined%
\gdef\SetFigFont#1#2#3#4#5{%
  \reset@font\fontsize{#1}{#2pt}%
  \fontfamily{#3}\fontseries{#4}\fontshape{#5}%
  \selectfont}%
\fi\endgroup%
\begin{picture}(7102,5127)(0,-12362)
\put(421,-10036){\rotatebox{90.0}{\makebox(0,0)[lb]{\smash{{\SetFigFont{12}{14.4}{\familydefault}{\mddefault}{\updefault}{\color[rgb]{0,0,0}Error}%
}}}}}
\put(2305,-12265){\makebox(0,0)[lb]{\smash{{\SetFigFont{12}{14.4}{\familydefault}{\mddefault}{\updefault}{\color[rgb]{0,0,0}Number of degrees of freedom ($N$)}%
}}}}
\put(5176,-8911){\makebox(0,0)[lb]{\smash{{\SetFigFont{10}{12.0}{\familydefault}{\mddefault}{\updefault}{\color[rgb]{0,0,0}$cN^{-0.275}$}%
}}}}
\put(5476,-9961){\makebox(0,0)[lb]{\smash{{\SetFigFont{10}{12.0}{\familydefault}{\mddefault}{\updefault}{\color[rgb]{0,0,0}$cN^{-0.5}$}%
}}}}
\put(1576,-11041){\makebox(0,0)[lb]{\smash{{\SetFigFont{10}{12.0}{\familydefault}{\mddefault}{\updefault}{\color[rgb]{0,0,0}$\vrho_{\Delta_\s}^{1/2}$ (uniform)}%
}}}}
\put(1576,-11557){\makebox(0,0)[lb]{\smash{{\SetFigFont{10}{12.0}{\familydefault}{\mddefault}{\updefault}{\color[rgb]{0,0,0}$\vrho_{\Delta_\s}^{1/2}$ (adaptive)}%
}}}}
\put(1576,-10765){\makebox(0,0)[lb]{\smash{{\SetFigFont{10}{12.0}{\familydefault}{\mddefault}{\updefault}{\color[rgb]{0,0,0}$\heta_{\Delta_\s}^h$ (uniform)}%
}}}}
\put(1576,-11305){\makebox(0,0)[lb]{\smash{{\SetFigFont{10}{12.0}{\familydefault}{\mddefault}{\updefault}{\color[rgb]{0,0,0}$\heta_{\Delta_\s}^h$ (adaptive)}%
}}}}
\end{picture}%

%% file: pLaplace_Lshape_1pt2_energy_vers2.pdf_t
\begin{picture}(0,0)%
\includegraphics{pLaplace_Lshape_1pt2_energy_vers2.pdf}%
\end{picture}%
\setlength{\unitlength}{3947sp}%
\begingroup\makeatletter\ifx\SetFigFont\undefined%
\gdef\SetFigFont#1#2#3#4#5{%
  \reset@font\fontsize{#1}{#2pt}%
  \fontfamily{#3}\fontseries{#4}\fontshape{#5}%
  \selectfont}%
\fi\endgroup%
\begin{picture}(7002,5252)(0,-12362)
\put(2251,-12286){\makebox(0,0)[lb]{\smash{{\SetFigFont{12}{14.4}{\familydefault}{\mddefault}{\updefault}{\color[rgb]{0,0,0}Number of degrees of freedom}%
}}}}
\put(301,-10501){\rotatebox{90.0}{\makebox(0,0)[lb]{\smash{{\SetFigFont{12}{14.4}{\familydefault}{\mddefault}{\updefault}{\color[rgb]{0,0,0}Primal and dual energy}%
}}}}}
\put(4501,-11317){\makebox(0,0)[lb]{\smash{{\SetFigFont{12}{14.4}{\familydefault}{\mddefault}{\updefault}{\color[rgb]{0,0,0}$E_{\Delta_\s}^h(u_h)$ (adaptive)}%
}}}}
\put(4501,-11113){\makebox(0,0)[lb]{\smash{{\SetFigFont{12}{14.4}{\familydefault}{\mddefault}{\updefault}{\color[rgb]{0,0,0}$\hD_{\Delta_\s}^h(p_h)$ (uniform)}%
}}}}
\put(4501,-10873){\makebox(0,0)[lb]{\smash{{\SetFigFont{12}{14.4}{\familydefault}{\mddefault}{\updefault}{\color[rgb]{0,0,0}$E_{\Delta_\s}^h(u_h)$ (uniform)}%
}}}}
\put(4501,-11557){\makebox(0,0)[lb]{\smash{{\SetFigFont{12}{14.4}{\familydefault}{\mddefault}{\updefault}{\color[rgb]{0,0,0}$\hD_{\Delta_\s}^h(p_h)$ (adaptive)}%
}}}}
\end{picture}%

%% file: pLaplace_Lshape_1pt6_iterations_scaled.pdf_t
\begin{picture}(0,0)%
\includegraphics{pLaplace_Lshape_1pt6_iterations_scaled.pdf}%
\end{picture}%
\setlength{\unitlength}{3947sp}%
\begingroup\makeatletter\ifx\SetFigFont\undefined%
\gdef\SetFigFont#1#2#3#4#5{%
  \reset@font\fontsize{#1}{#2pt}%
  \fontfamily{#3}\fontseries{#4}\fontshape{#5}%
  \selectfont}%
\fi\endgroup%
\begin{picture}(6902,5377)(0,-12362)
\put(2461,-12286){\makebox(0,0)[lb]{\smash{{\SetFigFont{12}{14.4}{\familydefault}{\mddefault}{\updefault}{\color[rgb]{0,0,0}Number of degrees of freedom}%
}}}}
\put(301,-10381){\rotatebox{90.0}{\makebox(0,0)[lb]{\smash{{\SetFigFont{12}{14.4}{\familydefault}{\mddefault}{\updefault}{\color[rgb]{0,0,0}Number of iterations}%
}}}}}
\put(1576,-7729){\makebox(0,0)[lb]{\smash{{\SetFigFont{10}{12.0}{\familydefault}{\mddefault}{\updefault}{\color[rgb]{0,0,0}Primal uniform}%
}}}}
\put(1576,-7921){\makebox(0,0)[lb]{\smash{{\SetFigFont{10}{12.0}{\familydefault}{\mddefault}{\updefault}{\color[rgb]{0,0,0}Dual uniform}%
}}}}
\put(1576,-8113){\makebox(0,0)[lb]{\smash{{\SetFigFont{10}{12.0}{\familydefault}{\mddefault}{\updefault}{\color[rgb]{0,0,0}Primal adaptive}%
}}}}
\put(1576,-8305){\makebox(0,0)[lb]{\smash{{\SetFigFont{10}{12.0}{\familydefault}{\mddefault}{\updefault}{\color[rgb]{0,0,0}Dual adaptive}%
}}}}
\end{picture}%

%% file: pLaplace_Lshape_1pt2_iterations.pdf_t
\begin{picture}(0,0)%
\includegraphics{pLaplace_Lshape_1pt2_iterations.pdf}%
\end{picture}%
\setlength{\unitlength}{3947sp}%
\begingroup\makeatletter\ifx\SetFigFont\undefined%
\gdef\SetFigFont#1#2#3#4#5{%
  \reset@font\fontsize{#1}{#2pt}%
  \fontfamily{#3}\fontseries{#4}\fontshape{#5}%
  \selectfont}%
\fi\endgroup%
\begin{picture}(6902,5377)(0,-12362)
\put(2461,-12286){\makebox(0,0)[lb]{\smash{{\SetFigFont{12}{14.4}{\familydefault}{\mddefault}{\updefault}{\color[rgb]{0,0,0}Number of degrees of freedom}%
}}}}
\put(301,-10381){\rotatebox{90.0}{\makebox(0,0)[lb]{\smash{{\SetFigFont{12}{14.4}{\familydefault}{\mddefault}{\updefault}{\color[rgb]{0,0,0}Number of iterations}%
}}}}}
\put(1576,-7729){\makebox(0,0)[lb]{\smash{{\SetFigFont{10}{12.0}{\familydefault}{\mddefault}{\updefault}{\color[rgb]{0,0,0}Primal uniform}%
}}}}
\put(1576,-7921){\makebox(0,0)[lb]{\smash{{\SetFigFont{10}{12.0}{\familydefault}{\mddefault}{\updefault}{\color[rgb]{0,0,0}Dual uniform}%
}}}}
\put(1576,-8113){\makebox(0,0)[lb]{\smash{{\SetFigFont{10}{12.0}{\familydefault}{\mddefault}{\updefault}{\color[rgb]{0,0,0}Primal adaptive}%
}}}}
\put(1576,-8305){\makebox(0,0)[lb]{\smash{{\SetFigFont{10}{12.0}{\familydefault}{\mddefault}{\updefault}{\color[rgb]{0,0,0}Dual adaptive}%
}}}}
\end{picture}%

%% file: pLaplace_Lshape_1pt6_estimators2.pdf_tex
\begingroup%
  \makeatletter%
  \providecommand\color[2][]{%
    \errmessage{(Inkscape) Color is used for the text in Inkscape, but the package 'color.sty' is not loaded}%
    \renewcommand\color[2][]{}%
  }%
  \providecommand\transparent[1]{%
    \errmessage{(Inkscape) Transparency is used (non-zero) for the text in Inkscape, but the package 'transparent.sty' is not loaded}%
    \renewcommand\transparent[1]{}%
  }%
  \providecommand\rotatebox[2]{#2}%
  \ifx\svgwidth\undefined%
    \setlength{\unitlength}{426bp}%
    \ifx\svgscale\undefined%
      \relax%
    \else%
      \setlength{\unitlength}{\unitlength * \real{\svgscale}}%
    \fi%
  \else%
    \setlength{\unitlength}{\svgwidth}%
  \fi%
  \global\let\svgwidth\undefined%
  \global\let\svgscale\undefined%
  \makeatother%
  \begin{picture}(1,0.72300469)%
    \put(0,0){\includegraphics[width=\unitlength]{pLaplace_Lshape_1pt6_estimators2.pdf}}%
    \put(0.10915493,0.04225352){\makebox(0,0)[lb]{\smash{10}}}%
    \put(0.13908451,0.0528169){\makebox(0,0)[lb]{\smash{2}}}%
    \put(0.49647887,0.04225352){\makebox(0,0)[lb]{\smash{10}}}%
    \put(0.52640845,0.0528169){\makebox(0,0)[lb]{\smash{3}}}%
    \put(0.88380282,0.04225352){\makebox(0,0)[lb]{\smash{10}}}%
    \put(0.91373239,0.0528169){\makebox(0,0)[lb]{\smash{4}}}%
    \put(0.0721831,0.06866197){\makebox(0,0)[lb]{\smash{10}}}%
    \put(0.10211268,0.07922535){\makebox(0,0)[lb]{\smash{-3}}}%
    \put(0.0721831,0.21478873){\makebox(0,0)[lb]{\smash{10}}}%
    \put(0.10211268,0.22535211){\makebox(0,0)[lb]{\smash{-2}}}%
    \put(0.0721831,0.36091549){\makebox(0,0)[lb]{\smash{10}}}%
    \put(0.10211268,0.37147887){\makebox(0,0)[lb]{\smash{-1}}}%
    \put(0.07922535,0.50880282){\makebox(0,0)[lb]{\smash{10}}}%
    \put(0.10915493,0.5193662){\makebox(0,0)[lb]{\smash{0}}}%
    \put(0.07922535,0.65492958){\makebox(0,0)[lb]{\smash{10}}}%
    \put(0.10915493,0.66549296){\makebox(0,0)[lb]{\smash{1}}}%
    \put(0.37667236,0.01411167){\color[rgb]{0,0,0}\makebox(0,0)[lb]{\smash{Number of degrees of freedom}}}%
    \put(0.04877773,0.34576384){\color[rgb]{0,0,0}\rotatebox{90}{\makebox(0,0)[lb]{\smash{Error}}}}%
    \put(0.21750162,0.18508476){\color[rgb]{0,0,0}\makebox(0,0)[lb]{\smash{\footnotesize{$\heta_{\Delta_\s}^h$ (adaptive)}}}}%
    \put(0.21858712,0.14957523){\color[rgb]{0,0,0}\makebox(0,0)[lb]{\smash{\footnotesize{$\eta_{res}^h$ (adaptive)}}}}%
    \put(0.22035876,0.11397866){\color[rgb]{0,0,0}\makebox(0,0)[lb]{\smash{\footnotesize{$\vrho_{\Delta_\s}^{1/2}$ (adaptive)}}}}%
  \end{picture}%
\endgroup%

%% file: pLaplace_Lshape_1pt2_estimators.pdf_tex
\begingroup%
  \makeatletter%
  \providecommand\color[2][]{%
    \errmessage{(Inkscape) Color is used for the text in Inkscape, but the package 'color.sty' is not loaded}%
    \renewcommand\color[2][]{}%
  }%
  \providecommand\transparent[1]{%
    \errmessage{(Inkscape) Transparency is used (non-zero) for the text in Inkscape, but the package 'transparent.sty' is not loaded}%
    \renewcommand\transparent[1]{}%
  }%
  \providecommand\rotatebox[2]{#2}%
  \ifx\svgwidth\undefined%
    \setlength{\unitlength}{420bp}%
    \ifx\svgscale\undefined%
      \relax%
    \else%
      \setlength{\unitlength}{\unitlength * \real{\svgscale}}%
    \fi%
  \else%
    \setlength{\unitlength}{\svgwidth}%
  \fi%
  \global\let\svgwidth\undefined%
  \global\let\svgscale\undefined%
  \makeatother%
  \begin{picture}(1,0.75)%
    \put(0,0){\includegraphics[width=\unitlength]{pLaplace_Lshape_1pt2_estimators.pdf}}%
    \put(0.10892857,0.04285714){\makebox(0,0)[lb]{\smash{10}}}%
    \put(0.13928571,0.05357143){\makebox(0,0)[lb]{\smash{2}}}%
    \put(0.49642857,0.04285714){\makebox(0,0)[lb]{\smash{10}}}%
    \put(0.52678571,0.05357143){\makebox(0,0)[lb]{\smash{3}}}%
    \put(0.88392857,0.04285714){\makebox(0,0)[lb]{\smash{10}}}%
    \put(0.91428571,0.05357143){\makebox(0,0)[lb]{\smash{4}}}%
    \put(0.07142857,0.06964286){\makebox(0,0)[lb]{\smash{10}}}%
    \put(0.10178571,0.08035714){\makebox(0,0)[lb]{\smash{-3}}}%
    \put(0.07142857,0.22321429){\makebox(0,0)[lb]{\smash{10}}}%
    \put(0.10178571,0.23392857){\makebox(0,0)[lb]{\smash{-2}}}%
    \put(0.07142857,0.375){\makebox(0,0)[lb]{\smash{10}}}%
    \put(0.10178571,0.38571429){\makebox(0,0)[lb]{\smash{-1}}}%
    \put(0.07857143,0.52857143){\makebox(0,0)[lb]{\smash{10}}}%
    \put(0.10892857,0.53928571){\makebox(0,0)[lb]{\smash{0}}}%
    \put(0.07857143,0.68214286){\makebox(0,0)[lb]{\smash{10}}}%
    \put(0.10892857,0.69285714){\makebox(0,0)[lb]{\smash{1}}}%
    \put(0.05654762,0.35446429){\rotatebox{90}{\makebox(0,0)[lb]{\smash{Error}}}}%
    \put(0.22177419,0.19455645){\color[rgb]{0,0,0}\makebox(0,0)[lb]{\smash{\footnotesize{$\heta_{\Delta_\s}^h$ (adaptive)}}}}%
    \put(0.2217742,0.15625001){\color[rgb]{0,0,0}\makebox(0,0)[lb]{\smash{\footnotesize{$\eta_{res}^h$ (adaptive)}}}}%
    \put(0.22278227,0.11592742){\color[rgb]{0,0,0}\makebox(0,0)[lb]{\smash{\footnotesize{$\vrho_{\Delta_\s}^{1/2}$ (adaptive)}}}}%
    \put(0.35685483,0.01008062){\color[rgb]{0,0,0}\makebox(0,0)[lb]{\smash{     Number of degrees of freedom}}}%
  \end{picture}%
\endgroup%

%% file: ROF_square_cont_error.pdf_t
\begin{picture}(0,0)%
\includegraphics{ROF_square_cont_error.pdf}%
\end{picture}%
\setlength{\unitlength}{3947sp}%
\begingroup\makeatletter\ifx\SetFigFont\undefined%
\gdef\SetFigFont#1#2#3#4#5{%
  \reset@font\fontsize{#1}{#2pt}%
  \fontfamily{#3}\fontseries{#4}\fontshape{#5}%
  \selectfont}%
\fi\endgroup%
\begin{picture}(7102,5127)(0,-12362)
\put(2461,-12286){\makebox(0,0)[lb]{\smash{{\SetFigFont{12}{14.4}{\familydefault}{\mddefault}{\updefault}{\color[rgb]{0,0,0}Number of degrees of freedom}%
}}}}
\put(376,-10321){\rotatebox{90.0}{\makebox(0,0)[lb]{\smash{{\SetFigFont{12}{14.4}{\familydefault}{\mddefault}{\updefault}{\color[rgb]{0,0,0}Error estimator}%
}}}}}
\put(4951,-9361){\makebox(0,0)[lb]{\smash{{\SetFigFont{10}{12.0}{\familydefault}{\mddefault}{\updefault}{\color[rgb]{0,0,0}$cN^{-0.24}$}%
}}}}
\put(4801,-11011){\makebox(0,0)[lb]{\smash{{\SetFigFont{10}{12.0}{\familydefault}{\mddefault}{\updefault}{\color[rgb]{0,0,0}$cN^{-0.38}$}%
}}}}
\put(1576,-11233){\makebox(0,0)[lb]{\smash{{\SetFigFont{10}{12.0}{\familydefault}{\mddefault}{\updefault}{\color[rgb]{0,0,0}$\eta_{ROF}^h$ (uniform)}%
}}}}
\put(1576,-11521){\makebox(0,0)[lb]{\smash{{\SetFigFont{10}{12.0}{\familydefault}{\mddefault}{\updefault}{\color[rgb]{0,0,0}$\eta_{ROF}^h$ (adaptive)}%
}}}}
\end{picture}%

%% file: ROF_square_disc_error.pdf_t
\begin{picture}(0,0)%
\includegraphics{ROF_square_disc_error.pdf}%
\end{picture}%
\setlength{\unitlength}{3947sp}%
\begingroup\makeatletter\ifx\SetFigFont\undefined%
\gdef\SetFigFont#1#2#3#4#5{%
  \reset@font\fontsize{#1}{#2pt}%
  \fontfamily{#3}\fontseries{#4}\fontshape{#5}%
  \selectfont}%
\fi\endgroup%
\begin{picture}(7102,5127)(0,-12362)
\put(451,-10321){\rotatebox{90.0}{\makebox(0,0)[lb]{\smash{{\SetFigFont{12}{14.4}{\familydefault}{\mddefault}{\updefault}{\color[rgb]{0,0,0}Error estimator}%
}}}}}
\put(2521,-12286){\makebox(0,0)[lb]{\smash{{\SetFigFont{12}{14.4}{\familydefault}{\mddefault}{\updefault}{\color[rgb]{0,0,0}Number of degrees of freedom}%
}}}}
\put(5026,-9361){\makebox(0,0)[lb]{\smash{{\SetFigFont{10}{12.0}{\familydefault}{\mddefault}{\updefault}{\color[rgb]{0,0,0}$cN^{-0.24}$}%
}}}}
\put(5026,-11086){\makebox(0,0)[lb]{\smash{{\SetFigFont{10}{12.0}{\familydefault}{\mddefault}{\updefault}{\color[rgb]{0,0,0}$cN^{-0.4}$}%
}}}}
\put(1576,-11233){\makebox(0,0)[lb]{\smash{{\SetFigFont{10}{12.0}{\familydefault}{\mddefault}{\updefault}{\color[rgb]{0,0,0}$\eta_{ROF}^h$ (uniform)}%
}}}}
\put(1576,-11521){\makebox(0,0)[lb]{\smash{{\SetFigFont{10}{12.0}{\familydefault}{\mddefault}{\updefault}{\color[rgb]{0,0,0}$\eta_{ROF}^h$ (adaptive)}%
}}}}
\end{picture}%

%% file: ROF_circle_cont_error.pdf_t
\begin{picture}(0,0)%
\includegraphics{ROF_circle_cont_error.pdf}%
\end{picture}%
\setlength{\unitlength}{3947sp}%
\begingroup\makeatletter\ifx\SetFigFont\undefined%
\gdef\SetFigFont#1#2#3#4#5{%
  \reset@font\fontsize{#1}{#2pt}%
  \fontfamily{#3}\fontseries{#4}\fontshape{#5}%
  \selectfont}%
\fi\endgroup%
\begin{picture}(7102,5127)(0,-12362)
\put(2281,-12286){\makebox(0,0)[lb]{\smash{{\SetFigFont{12}{14.4}{\familydefault}{\mddefault}{\updefault}{\color[rgb]{0,0,0}Number of degrees of freedom ($N$)}%
}}}}
\put(385,-9901){\rotatebox{90.0}{\makebox(0,0)[lb]{\smash{{\SetFigFont{12}{14.4}{\familydefault}{\mddefault}{\updefault}{\color[rgb]{0,0,0}Error}%
}}}}}
\put(5551,-8761){\makebox(0,0)[lb]{\smash{{\SetFigFont{10}{12.0}{\familydefault}{\mddefault}{\updefault}{\color[rgb]{0,0,0}$cN^{-0.22}$}%
}}}}
\put(5551,-9436){\makebox(0,0)[lb]{\smash{{\SetFigFont{10}{12.0}{\familydefault}{\mddefault}{\updefault}{\color[rgb]{0,0,0}$cN^{-0.29}$}%
}}}}
\put(1576,-11005){\makebox(0,0)[lb]{\smash{{\SetFigFont{10}{12.0}{\familydefault}{\mddefault}{\updefault}{\color[rgb]{0,0,0}$\vrho_{ROF}^{1/2}$ (uniform)}%
}}}}
\put(1576,-11545){\makebox(0,0)[lb]{\smash{{\SetFigFont{10}{12.0}{\familydefault}{\mddefault}{\updefault}{\color[rgb]{0,0,0}$\vrho_{ROF}^{1/2}$ (adaptive)}%
}}}}
\put(1576,-10741){\makebox(0,0)[lb]{\smash{{\SetFigFont{10}{12.0}{\familydefault}{\mddefault}{\updefault}{\color[rgb]{0,0,0}$\eta_{ROF}^h$ (uniform)}%
}}}}
\put(1576,-11269){\makebox(0,0)[lb]{\smash{{\SetFigFont{10}{12.0}{\familydefault}{\mddefault}{\updefault}{\color[rgb]{0,0,0}$\eta_{ROF}^h$ (adaptive)}%
}}}}
\end{picture}%

%% file: ROF_circle_disc_error.pdf_t
\begin{picture}(0,0)%
\includegraphics{ROF_circle_disc_error.pdf}%
\end{picture}%
\setlength{\unitlength}{3947sp}%
\begingroup\makeatletter\ifx\SetFigFont\undefined%
\gdef\SetFigFont#1#2#3#4#5{%
  \reset@font\fontsize{#1}{#2pt}%
  \fontfamily{#3}\fontseries{#4}\fontshape{#5}%
  \selectfont}%
\fi\endgroup%
\begin{picture}(7102,5127)(0,-12362)
\put(2281,-12286){\makebox(0,0)[lb]{\smash{{\SetFigFont{12}{14.4}{\familydefault}{\mddefault}{\updefault}{\color[rgb]{0,0,0}Number of degrees of freedom ($N$)}%
}}}}
\put(385,-9901){\rotatebox{90.0}{\makebox(0,0)[lb]{\smash{{\SetFigFont{12}{14.4}{\familydefault}{\mddefault}{\updefault}{\color[rgb]{0,0,0}Error}%
}}}}}
\put(5551,-9436){\makebox(0,0)[lb]{\smash{{\SetFigFont{10}{12.0}{\familydefault}{\mddefault}{\updefault}{\color[rgb]{0,0,0}$cN^{-0.29}$}%
}}}}
\put(5551,-8761){\makebox(0,0)[lb]{\smash{{\SetFigFont{10}{12.0}{\familydefault}{\mddefault}{\updefault}{\color[rgb]{0,0,0}$cN^{-0.22}$}%
}}}}
\put(1576,-11005){\makebox(0,0)[lb]{\smash{{\SetFigFont{10}{12.0}{\familydefault}{\mddefault}{\updefault}{\color[rgb]{0,0,0}$\vrho_{ROF}^{1/2}$ (uniform)}%
}}}}
\put(1576,-11545){\makebox(0,0)[lb]{\smash{{\SetFigFont{10}{12.0}{\familydefault}{\mddefault}{\updefault}{\color[rgb]{0,0,0}$\vrho_{ROF}^{1/2}$ (adaptive)}%
}}}}
\put(1576,-10741){\makebox(0,0)[lb]{\smash{{\SetFigFont{10}{12.0}{\familydefault}{\mddefault}{\updefault}{\color[rgb]{0,0,0}$\eta_{ROF}^h$ (uniform)}%
}}}}
\put(1576,-11269){\makebox(0,0)[lb]{\smash{{\SetFigFont{10}{12.0}{\familydefault}{\mddefault}{\updefault}{\color[rgb]{0,0,0}$\eta_{ROF}^h$ (adaptive)}%
}}}}
\end{picture}%

%% file: ROF_circle_cont_iterations.pdf_t
\begin{picture}(0,0)%
\includegraphics{ROF_circle_cont_iterations.pdf}%
\end{picture}%
\setlength{\unitlength}{3947sp}%
\begingroup\makeatletter\ifx\SetFigFont\undefined%
\gdef\SetFigFont#1#2#3#4#5{%
  \reset@font\fontsize{#1}{#2pt}%
  \fontfamily{#3}\fontseries{#4}\fontshape{#5}%
  \selectfont}%
\fi\endgroup%
\begin{picture}(7102,5127)(0,-12362)
\put(2461,-12286){\makebox(0,0)[lb]{\smash{{\SetFigFont{12}{14.4}{\familydefault}{\mddefault}{\updefault}{\color[rgb]{0,0,0}Number of degrees of freedom}%
}}}}
\put(301,-10561){\rotatebox{90.0}{\makebox(0,0)[lb]{\smash{{\SetFigFont{12}{14.4}{\familydefault}{\mddefault}{\updefault}{\color[rgb]{0,0,0}Number of iterations}%
}}}}}
\put(1576,-7957){\makebox(0,0)[lb]{\smash{{\SetFigFont{10}{12.0}{\familydefault}{\mddefault}{\updefault}{\color[rgb]{0,0,0}Primal uniform}%
}}}}
\put(1576,-8149){\makebox(0,0)[lb]{\smash{{\SetFigFont{10}{12.0}{\familydefault}{\mddefault}{\updefault}{\color[rgb]{0,0,0}Dual uniform}%
}}}}
\put(1576,-8341){\makebox(0,0)[lb]{\smash{{\SetFigFont{10}{12.0}{\familydefault}{\mddefault}{\updefault}{\color[rgb]{0,0,0}Primal adaptive}%
}}}}
\put(1576,-8533){\makebox(0,0)[lb]{\smash{{\SetFigFont{10}{12.0}{\familydefault}{\mddefault}{\updefault}{\color[rgb]{0,0,0}Dual adaptive}%
}}}}
\end{picture}%

%% file: ROF_circle_disc_iterations.pdf_t
\begin{picture}(0,0)%
\includegraphics{ROF_circle_disc_iterations.pdf}%
\end{picture}%
\setlength{\unitlength}{3947sp}%
\begingroup\makeatletter\ifx\SetFigFont\undefined%
\gdef\SetFigFont#1#2#3#4#5{%
  \reset@font\fontsize{#1}{#2pt}%
  \fontfamily{#3}\fontseries{#4}\fontshape{#5}%
  \selectfont}%
\fi\endgroup%
\begin{picture}(7002,5252)(0,-12362)
\put(2401,-12286){\makebox(0,0)[lb]{\smash{{\SetFigFont{12}{14.4}{\familydefault}{\mddefault}{\updefault}{\color[rgb]{0,0,0}Number of degrees of freedom}%
}}}}
\put(301,-10501){\rotatebox{90.0}{\makebox(0,0)[lb]{\smash{{\SetFigFont{12}{14.4}{\familydefault}{\mddefault}{\updefault}{\color[rgb]{0,0,0}Number of iterations}%
}}}}}
\put(1576,-7837){\makebox(0,0)[lb]{\smash{{\SetFigFont{10}{12.0}{\familydefault}{\mddefault}{\updefault}{\color[rgb]{0,0,0}Primal uniform}%
}}}}
\put(1576,-8029){\makebox(0,0)[lb]{\smash{{\SetFigFont{10}{12.0}{\familydefault}{\mddefault}{\updefault}{\color[rgb]{0,0,0}Dual uniform}%
}}}}
\put(1576,-8221){\makebox(0,0)[lb]{\smash{{\SetFigFont{10}{12.0}{\familydefault}{\mddefault}{\updefault}{\color[rgb]{0,0,0}Primal adaptive}%
}}}}
\put(1576,-8413){\makebox(0,0)[lb]{\smash{{\SetFigFont{10}{12.0}{\familydefault}{\mddefault}{\updefault}{\color[rgb]{0,0,0}Dual adaptive}%
}}}}
\end{picture}%